\documentclass{amsart}

\usepackage{graphicx}
\usepackage{palatino}
\usepackage{parskip}
\usepackage{amsthm,amsfonts,amsmath,amssymb,amscd}
\usepackage[all,cmtip]{xy}
\usepackage[backref]{hyperref}
\usepackage{verbatim}
\usepackage{enumerate}

\newtheorem{mainthm}{Theorem}
\newtheorem{maincor}[mainthm]{Corollary}
\newtheorem{thm}{Theorem}[subsection]
\newtheorem{lma}[thm]{Lemma}
\newtheorem{cor}[thm]{Corollary}
\newtheorem{prp}[thm]{Proposition}

\theoremstyle{remark}
\newtheorem{rmk}[thm]{Remark}

\theoremstyle{definition}
\newtheorem{dfn}[thm]{Definition}
\newtheorem{exm}[thm]{Example}

\newcommand{\pin}{{Pin}_-(2,\CC)}
\newcommand{\CC}{\mathbf{C}}
\newcommand{\PP}{\mathbf{P}}
\newcommand{\RR}{\mathbf{R}}
\newcommand{\ZZ}{\mathbf{Z}}
\newcommand{\OP}{\operatorname}
\newcommand{\Sym}{\OP{Sym}}
\newcommand{\SL}{SL(2,\CC)}
\newcommand{\GL}{GL}
\newcommand{\SU}{SU}

\newcommand{\cp}[1]{\CC\PP^{#1}}
\newcommand{\rp}[1]{\RR\PP^{#1}}
\newcommand{\dbar}{\overline{\partial}}
\newcommand{\mM}{\mathcal{M}}
\newcommand{\ev}{\OP{ev}}
\newcommand{\kk}{\mathbf{F}}

\title{Floer cohomology of the Chiang Lagrangian}
\author[Jonathan David Evans and Yank{\i} Lekili]{Jonathan David Evans\\ Yank{\i} Lekili}
\address{University College London}
\address{King's College London}

\begin{document}

\maketitle
\begin{abstract}

We study holomorphic discs with boundary on a Lagrangian submanifold $L$ in a
K\"{a}hler manifold admitting a Hamiltonian action of a group $K$ which has $L$
as an orbit. We prove various transversality and classification results for
such discs which we then apply to the case of a particular Lagrangian in
$\CC\PP^3$ first noticed by Chiang \cite{Chiang}. We prove that this Lagrangian
has non-vanishing Floer cohomology if and only if the coefficient ring has
characteristic 5, in which case it generates the
split-closed derived Fukaya category as a triangulated category.



\end{abstract}

\section{Introduction}

\subsection{Homogeneous Lagrangian submanifold}

In recent years there has been much interest in the symplectic geometry of toric manifolds and in the Lagrangian Floer theory of toric fibres \cite{fooo-toric1,fooo-toric2}. These toric fibres are the simplest {\em homogeneous Lagrangian submanifolds}:

\begin{dfn}
Let $X$ be a smooth complex projective variety, $L\subset X$ be a Lagrangian submanifold (with respect to the restriction of the Fubini-Study form) and $K$ be a compact connected Lie group. Let $G$ be the complexification of $K$. Assume that:
\begin{itemize}
\item $G$ acts algebraically on $X$,
\item the restriction of this action to $K$ is a Hamiltonian action,
\item $L$ is a $K$-orbit.
\end{itemize}
We say that $(X,L)$ is $K$-homogeneous.
\end{dfn}

It is natural to wonder if the Lagrangian Floer theory of $K$-homogeneous Lagrangians displays the same richness as that of toric fibres when the group $K$ is allowed to be nonabelian. In this paper we make some inroads into the theory.

In the first half of the paper we prove transversality and classification results for holomorphic discs with boundary on a $K$-homogeneous Lagrangian. In particular we show that all holomorphic discs are regular and that the stabiliser of a Maslov 2 disc is one dimension larger than the stabiliser of a point.

\subsection{The Chiang Lagrangian}

We then introduce a family of four examples of monotone $SU(2)$-homogeneous
Lagrangians in quasihomogeneous Fano 3-folds of $SL(2,\CC)$. These 3-folds
arise by taking the closure of the $SL(2,\CC)$-orbit of a point configuration
$C$ in $\cp{1}$ where the configuration is one of: $C=\Delta$, an equilateral
triangle on the equator; $C=T$, the vertices of a tetrahedron; $C=O$, the
vertices of an octahedron; $C=I$ the vertices of an icosahedron. In each case
the Lagrangian is the $SU(2)$-orbit of the configuration $C$.

The first of these examples is the {\em Chiang Lagrangian} $L_{\Delta}\subset\cp{3}$, described in \cite{Chiang}. Topologically, $L_{\Delta}$ is the quotient of $SU(2)$ by the binary dihedral subgroup of order twelve. In particular it is a rational homology sphere with first homology $H_1(L_{\Delta};\ZZ)=\ZZ/4$. It is a monotone Lagrangian submanifold with minimal Maslov number 2.

We study the Chiang Lagrangian in detail using methods inspired by Hitchin's paper \cite{hitchin} on Poncelet polygons. We expect that the methods we employ should generalise to the examples associated to higher Platonic solids, just as Hitchin's do \cite{hitch2,hitch3}, but we defer their study for future work.

Our main results on the Chiang Lagrangian can be summarised as follows. Recall that Floer cohomology is a $\ZZ/2$-graded vector space which can be equipped with the structures of a $\ZZ/2$-graded ring or a $\ZZ/2$-graded $A_{\infty}$-algebra (by applying homological perturbation to the cochain-level $A_{\infty}$-algebra). Recall also that the count of Maslov 2 holomorphic discs with boundary on $L_{\Delta}$ passing through a fixed point $x\in L_{\Delta}$ is denoted $\mathfrak{m}_0(L_{\Delta})$.

\begin{mainthm}
Let $L_{\Delta}\subset\cp{3}$ denote the Chiang Lagrangian. Equip $L_{\Delta}$ with an orientation and a spin structure.
\begin{enumerate}
\item[(a)] (Lemma 6.1.2) We have $\mathfrak{m}_0(L_{\Delta})=\pm 3$. 
\item[(b)] (Corollary 7.2.5) Let $\kk$ be a field of characteristic 5. Equip $L_{\Delta}$ with a $(\ZZ/5)^\times$-local system $\zeta$. Its Floer cohomology is well-defined and
\[HF^0((L_{\Delta},\zeta),(L_{\Delta},\zeta);\kk)\cong HF^1((L_{\Delta},\zeta),(L_{\Delta},\zeta);\kk)\cong\kk.  \]
\item[(c)] (Theorem 8.2.2) The Floer cohomology ring is a Clifford algebra
\[HF^*((L_{\Delta},\zeta),(L_{\Delta},\zeta);\kk)\cong\kk[x]/(x^2-\zeta^3)\]
where $x$ has degree 1. 
\item[(d)] (Theorem 8.2.2) As an $A_{\infty}$-algebra, $HF^*((L_{\Delta},\zeta),(L_{\Delta},\zeta);\kk)$ is formal. 
\item[(e)] (Corollary 10.0.3) Moreover the four Lagrangian branes obtained by equipping the Chiang Lagrangian with the four possible $(\ZZ/5)^{\times}$-local systems generate the Fukaya category of $\cp{3}$ over $\kk$. 
\item[(f)] (Corollary 7.2.5) Over a field $\mathbf{K}$ of characteristic $p\neq 5$ we have \[ HF^*(L_{\Delta},L_{\Delta};\mathbf{K})=0 \] 
\end{enumerate}
\end{mainthm}

The theorem is proved by an explicit computation. We use the Biran-Cornea pearl complex to compute the Floer cohomology: we write down a Morse function (and use the standard complex structure) and enumerate all the pearly trajectories that contribute to the Floer differential.

\begin{rmk}
The assumption on the characteristic of $\kk$ is a little unusual but seems less surprising if we argue as follows. Floer cohomology can only be non-vanishing if $\mathfrak{m}_0(L_{\Delta})$ is an eigenvalue of the quantum multiplication map
\[c_1(\cp{3})\star\colon QH^*(\cp{3})\to QH^*(\cp{3}).\]
The characteristic polynomial of this map is $\lambda^4-256$ so we must work over a field of characteristic $p$ where
\[3^4-256=-5^2\times 7\equiv 0\mod p.\]
\end{rmk}

\begin{rmk}
The Floer cohomology of the Clifford torus $T_{Cl}\subset\cp{3}$ is a Clifford algebra, so the Floer cohomology of the pair $(L_{\Delta},T_{Cl})$ (both equipped with suitable $(\ZZ/5)^{\times}$-local systems) is a Clifford module. In Corollary \ref{pair} we identify this with the four-dimensional spin representation which implies (see Corollary \ref{cor:summand}) that $L_{\Delta}$ is an idempotent summand of the Clifford torus in the Fukaya category.
\end{rmk}
\begin{rmk} The ring structure on $HF((L_{\Delta},\zeta),(L_{\Delta},\zeta);\kk)$ is determined indirectly by a Ho\-ch\-schild cohomology computation, inspired by \cite{Smith}, and by identifying the Clifford module structure as in the previous remark. Note that when $\kk=\ZZ/5$ there are two distinct isomorphism classes of nondegenerate Clifford algebra $\kk[x]/(x^2-\zeta^3)$ according to whether $\zeta^3$ is a square modulo 5; to see this, note that for $\ZZ/2$-grading reasons the algebra must have the form $x^2=b$ for some $0\neq b\in\kk$ and under a change of coordinates $x\mapsto \tilde{x}=ax$ this becomes $\tilde{x}^2=a^2b$ so $b$ is determined up to multiplication by a square. For more on Clifford algebras in arbitrary characteristic see \cite{Chev}.\end{rmk}

\begin{rmk}
Note that we have an additive isomorphism $HF^*(L_{\Delta},L_{\Delta};\kk)\cong H^*(L_{\Delta};\kk)$ when the grading on cohomology is collapsed to a $\ZZ/2$-grading. We use the Biran-Cornea pearl complex to compute $HF^*$ so the Floer cochains are the critical points of a Morse function. The Floer differential is quite nontrivial (see Lemma \ref{lma:floer-cpx}). For example, classically (over $\kk$) the cohomology is generated by the minimum and the maximum of the Morse function, but the Floer cochain corresponding to the maximum of our chosen Morse function is not even coclosed; to find a Floer cocycle representing the nontrivial class of odd degree one must take a combination of index 1 and index 3 critical points.
\end{rmk}

The results above imply immediately that:

\begin{maincor}
The Chiang Lagrangian is not displaceable from itself or from the Clifford torus via Hamiltonian isotopies.
\end{maincor}

\begin{rmk}
Note that $L_{\Delta}$ and $\RR\PP^3$ intersect along a pair of circles in their standard positions and it is an interesting open question if they can be displaced from one another. Standard techniques in Floer theory cannot answer this question because $HF(L_{\Delta},\RR\PP^3)$ is not well-defined: Floer cohomology can only be defined for Lagrangians with the same $\mathfrak{m}_0$-value and $\mathfrak{m}_0(\RR\PP^3)=0$ as $\RR\PP^3$ has minimal Maslov 4.
\end{rmk}

\subsection{Acknowledgements}

J.E. would like to thank Jason Lotay for pointing out to him Hitchin's papers on Platonic solids. Both authors would like to thank Ed Segal for helpful discussions on Clifford modules. Y.L. is supported by a Royal Society Fellowship. Figure \ref{fig:prism} was produced using Fritz Obermeyer's software \href{http://www.jenn3d.org}{Jenn3d}.

\part{Holomorphic discs on homogeneous Lagrangian submanifolds}

\section{Riemann-Hilbert problems}\label{sct-riemann-hilbert}

Let $D=\{z\in\CC\ :\ |z|\leq 1\}$ denote the unit disc, $\partial D$ its boundary and $D^o=D\setminus\partial D$.

\subsection{Riemann-Hilbert problems in Lagrangian Floer theory}

\begin{dfn}
A {\em Riemann-Hilbert pair} consists of a holomorphic rank $n$ vector bundle $E\to D$ over the disc with an analytic totally real $n$-dimensional subbundle $F\subset E|_{\partial D}$.
\end{dfn}

\begin{rmk}
Theorem 3.3.13 in \cite{KatzLiu} shows that whenever one has a smooth complex vector bundle $E\to D$, holomorphic over the interior $D\setminus\partial D$, and a smooth totally real $n$-dimensional subbundle $F\subset E|_{\partial D}$ one can extend the holomorphic structure to the whole of $E$ to make it into a Riemann-Hilbert pair.
\end{rmk}

Given a Riemann-Hilbert pair $(E,F)$ and a number $p>2$, let $L^p_1(E,F)$ denote the $L^p_1$-Sobolev completion of the space of smooth sections with totally real boundary conditions and let $L^p(\Lambda^{0,1}\otimes E)$ denote the $L^p$-completion of the space of smooth $(0,1)$-forms with values in $E$. The holomorphic structure gives a Cauchy-Riemann operator
\[\dbar\colon L^p_1(E,F)\to L^p(\Lambda^{0,1}\otimes E)\]
which takes a smooth section $\sigma$ to $\dbar\sigma=\frac{1}{2}\left(d\sigma+J\circ d\sigma \circ i\right)$.

\begin{rmk}
This is a Fredholm operator. The kernel $\ker\dbar$ consists of holomorphic sections $\sigma$, $\dbar\sigma=0$, with totally real boundary conditions.
\end{rmk}

Riemann-Hilbert pairs arise in the following way in Lagrangian Floer theory.

\begin{dfn}
Let $(X,J)$ be a complex $n$-manifold, $L\subset X$ a smooth totally real $n$-dimensional submanifold and $u\colon (D,\partial D)\to (X,L)$ a $J$-holomorphic disc with boundary on $L$. We get a holomorphic vector bundle $u^*TX$ over $D$ and a smooth totally real subbundle $F=u|_{\partial D}^*TL\subset E|_{\partial D}$.
\end{dfn}

The importance of this Riemann-Hilbert pair is that the associated Cauchy-Riemann operator is the linearisation at $u$ of the holomorphic curve equation $\dbar_Ju=0$. If the cokernel of the Cauchy-Riemann operator vanishes then $\ker\dbar$ is the tangent space to the space of parametrised $J$-holomorphic discs at $u$.

\subsection{Oh's splitting theorem}

A holomorphic vector bundle over the disc is trivial, so there exists a smooth bundle trivialisation $\Psi\colon E\to D\times\CC^n$ holomorphic over $E|_{D^o}$. Under this trivialisation each space $F_z$, $z\in\partial D$, is identified with a totally real subspace of $\CC^n$.

The group $\GL(n,\CC)$ acts transitively on $n$-dimensional totally real subspaces with stabiliser $\GL(n,\RR)$, so $F$ defines a loop $\gamma\colon\partial D\to\GL(n,\CC)/\GL(n,\RR)$ by $\gamma(z)=\Psi(F_z)$. The fundamental group of this homogeneous space is $\ZZ$ and the winding number of our loop is called the Maslov number, $\mu(F)$, of the boundary condition $F$. Note that $\GL(n,\RR)$ has two components and the loop lifts to a loop in $\GL(n,\CC)$ if and only if $\mu(F)\equiv 0\mod 2$; however, we can always lift to a multivalued loop of matrices. Using a special form of Birkhoff factorisation proved by Globevnik {\cite[Lemma 5.1]{Glob}}, building on work of Vekua \cite{Vek},  Oh \cite{Oh} proved that we can find a holomorphic trivialisation $\Psi'$ for which the totally real boundary condition looks particularly simple.

\begin{thm}[{\cite[Theorem 1]{Oh}}]\label{thm-oh-splitting}
If $\gamma\colon\partial D\to\GL(n,\CC)/\GL(n,\RR)$ is a smooth loop of totally real subspaces then
\[\gamma(z)=\Theta(z)\Lambda^{1/2}(z)\cdot\RR^n,\ z\in\partial D,\]
where $\Theta\colon\partial D\to\GL(n,\CC)$ extends to a smooth map $D\to\GL(n,\CC)$ holomorphic on $D^o$ and
\[\Lambda(z)=\left(\begin{array}{ccc}z^{\kappa_1} & & 0\\
& \ddots & \\
0& &z^{\kappa_n}\end{array}\right)\]
for some integers $\kappa_i$ called the partial indices of $\gamma$. If some $\kappa_i$ is odd then $\Lambda^{1/2}(z)$ becomes double-valued.
\end{thm}

The holomorphic trivialisation in question is the composition of $\Psi$ with the fibrewise multiplication by $\Theta(z)^{-1}$. In this trivialisation the totally real boundary condition at $z\in\partial D$ is given by $\Lambda^{1/2}(z)\cdot\RR^n$. In particular, we see that a one-dimensional Riemann-Hilbert pair is completely classified up to isomorphism by its Maslov number and that the Riemann-Hilbert pair $(E,F)$ separates as a direct sum of one-dimensional Riemann-Hilbert pairs $(E_i,F_i)$ whose Maslov numbers $\kappa_i$ are the partial indices of the loop of totally real subspaces given by $F$.

\begin{dfn}
If $(E,F)$ is a Riemann-Hilbert pair which splits as a direct sum $\bigoplus_i(E_i,F_i)$ then we call the $(E_i,F_i)$ the {\em Riemann-Hilbert summands} of $(E,F)$.
\end{dfn}

The following proposition is proved by explicitly solving the $\dbar$-problem for the Riemann-Hilbert pair using Fourier theory with half-integer exponents.

\begin{thm}[{\cite[Propositions 5.1, 5.2, Theorem 5.3]{Oh}}]\label{thm-oh-sections}
Let $(E,F)$ be a one-dimensional Riemann-Hilbert pair and let $\kappa=\mu(F)$ be the Maslov number of $F$. If $\kappa\leq -1$ then
\[\dim\ker\dbar=0,\qquad\dim\OP{coker}\dbar=-\kappa-1.\]
If $\kappa\geq 0$ then
\[\dim\ker\dbar=\kappa+1,\qquad\dim\OP{coker}\dbar=0.\]
In particular the index of $\dbar$ is $\mu(F)+1$. If $(E,F)$ has dimension $n$ then the index of the corresponding $\dbar$-operator is the sum of the indices for its Riemann-Hilbert summands, namely
\[\mu(F)+n.\]
\end{thm}

\begin{rmk}\label{rmk-oh-zeros}
Suppose that $(E,F)$ is a one-dimensional Riemann-Hilbert pair with Maslov number $\kappa=\mu(F)$.
\begin{itemize}
\item If $\kappa$ is odd then any global section must vanish at some point in $\partial D$ because the total space of the totally real boundary condition is a M\"{o}bius strip in that case.
\item If there is a nowhere-vanishing global section then $\kappa=0$; conversely if $\kappa=0$ then any global section is either nowhere-vanishing or identically zero.
\end{itemize}
\end{rmk}

\subsection{Regularity}

\begin{dfn}
A Riemann-Hilbert pair is called regular if $\OP{coker}\dbar=0$.
\end{dfn}

It follows from Oh's theorems above that a Riemann-Hilbert pair is regular if and only if all of its partial indices $\kappa_i$ satisfy $\kappa_i\geq -1$. In general it is not easy to control these partial indices for the Riemann-Hilbert pairs arising in Lagrangian Floer theory. In the cases we are studying we will use the presence of symmetry to prove that the Riemann-Hilbert pair satisfies the following criterion, which in turn implies that the partial indices are all nonnegative.

\begin{dfn}
A Riemann-Hilbert pair $(E,F)$ is generated by global sections at a point of the boundary if there is a point $z\in\partial D$ such that the evaluation map $\OP{ev}_z\colon\ker\dbar\to F_z$, which sends $\sigma$ to $\sigma(z)$, is surjective.
\end{dfn}

A Riemann-Hilbert pair splits into its Riemann-Hilbert summands $(E_i,F_i)$ and the evaluation map becomes block-diagonal $\OP{ev}_z\colon\bigoplus_i\ker\dbar_{(E_i,F_i)}\to \bigoplus_i(F_i)_z$. In particular, if $(E,F)$ is generated by global sections at $z\in\partial D$ then the same is true of its Riemann-Hilbert summands.

\begin{lma}\label{lma-globally-generated}
If $(E,F)$ is generated by global sections at a point of the boundary then its partial indices are all nonnegative. In particular, $(E,F)$ is regular.
\end{lma}
\begin{proof}
Since the Riemann-Hilbert summands are generated by global sections at $z\in\partial D$ they admit global sections. By Theorem \ref{thm-oh-sections}, the only one-dimensional Riemann-Hilbert pairs with global sections are those with nonnegative Maslov number.
\end{proof}

When studying transversality of evaluation maps in Lagrangian Floer theory we will need the following result:

\begin{lma}\label{lma-transverse-twopoint}
Fix a pair of distinct points $z_1,z_2\in\partial D$. If $(E,F)$ is an $n$-dimensional Riemann-Hilbert pair with $\mu(F)=n$ whose partial indices are $\kappa_1=1,\ldots,\kappa_n=1$ then the evaluation map
\[\OP{ev}_{z_1,z_2}\colon\ker\dbar\to F_{z_1}\oplus F_{z_2}\]
sending $\sigma$ to $(\sigma(z_1),\sigma(z_2))$ is surjective.
\end{lma}
\begin{proof}
It suffices to prove surjectivity for a single Riemann-Hilbert summand so we assume $n=1$. We work with Oh's trivialisation so that the boundary condition is given by
\[F=z^{1/2}\cdot\RR^n.\]
Oh {\cite[Section 5, Case II]{Oh}} proves that the only global sections are of the form $cz+\bar{c}$. If $c\neq 0$, these sections have a single zero at $-\bar{c}/c\in\partial D$. In particular there exist sections $\sigma_1$ and $\sigma_2$ such that $\sigma_i$ vanishes precisely at $z_i$ for $i=1,2$. The images of these sections under $\OP{ev}_{z_1,z_2}$ span $F_{z_1}\oplus F_{z_2}$.
\end{proof}

\section{Holomorphic discs with symmetry}

\subsection{Overview}

In this section we will study the Riemann-Hilbert pairs associated to holomorphic discs on homogeneous Lagrangians and find several applications of the theory from Section \ref{sct-riemann-hilbert} to Lagrangian Floer theory. In this section $X$ is a smooth complex variety of dimension $n$ with complex structure $J$ and $L\subset X$ is a totally real submanifold. We assume that $X$ admits an action of a compact connected Lie group $K$ which extends to an algebraic action of the complexification $G$ and for which $L$ is a $K$-orbit. We continue to use the name $K$-homogeneous for this slightly weaker set of assumptions: the symplectic structure and Lagrangian conditions are not important in this section. All holomorphic discs are assumed to be non-constant.

It will be convenient to make the following definition.

\begin{dfn}\label{dfn:half-maslov}
If $(X,L)$ is $K$-homogeneous, a {\em half-Maslov divisor} is a $G$-invariant divisor $Y\subset X\setminus L$ such that the Maslov index of a Riemann surface $u$ with boundary on $L$ equals $2[u]\cdot [Y]$.
\end{dfn}

\subsection{Moduli spaces of $J$-holomorphic discs}

Let $\dbar_Ju=0$ denote the nonlinear Cauchy-Riemann equation whose solutions are $J$-holomorphic maps
\[u\colon(D,\partial D)\to(X,L).\]
Fix a relative homology class $0\neq\beta\in H_2(X,L;\ZZ)$. We define the moduli spaces
\begin{align*}\mM^*_{0,k}(J,\beta)=\left\{(u,z_1,\ldots,z_k)\ \right.&:\  u\colon(D,\partial D)\to(X,L),\ u\ \mbox{somewhere injective}\\
&\ \ \ \ \left. [u]=\beta,\ \dbar_Ju=0,\ z_i\in\partial D,\ z_i\neq z_j\right\}/\sim\end{align*}
where $\sim$ is the relation
\[(u,z_1,\ldots,z_k)\sim (u\circ\phi^{-1},\phi(z_1),\ldots,\phi(z_k))\]
for some $\phi\in PSL(2,\RR)$, the holomorphic automorphism group of the disc. Note that $PSL(2,\RR)$ acts freely on the space of somewhere injective discs.

If the Riemann-Hilbert pair associated to $u$ is regular then the moduli space is a smooth manifold in a neighbourhood of $u$ and its tangent space is
\[T_{[u,z_1,\ldots,z_k]}\mM^*_{0,k}(J,\beta)=\left(\ker\dbar\oplus T_{z_1}\partial D\oplus\cdots\oplus T_{z_k}\partial D\right)/\mathfrak{psl}(2,\RR)\]
where $\mathfrak{psl}(2,\RR)$ denotes the infinitesimal action of automorphisms. This has dimension $\mu+n+k-3$.

\subsection{Symmetry implies regularity}

\begin{lma}\label{lma-reg}
If $(X,L)$ is $K$-homogeneous and $u\colon(D,\partial D)\to(X,L)$ is a $J$-holomorphic disc then the associated Riemann-Hilbert pair is generated by global sections, and hence regular. As a consequence, if $(X,L)$ is $K$-homogeneous then all moduli spaces $\mM^*_{0,k}(J,\beta)$ of $J$-holomorphic discs with boundary on $L$ are smooth manifolds.
\end{lma}
\begin{proof}
Each element of the Lie algebra $\mathfrak{k}$ of $K$ defines a holomorphic vector field on $X$ which is tangent to $L$ along $L$, in particular there is a map $\mathfrak{k}\to\ker\dbar$ where $\dbar$ is the Cauchy-Riemann operator for the Riemann-Hilbert pair associated to $u$. For any point $z\in\partial D$ there is a surjective map $\mathfrak{k}\to T_{u(z)}L$ coming from the evaluation of these holomorphic vector fields at the point $u(z)$. Therefore the Riemann-Hilbert pair is generated by global sections at $z$, so by Lemma \ref{lma-globally-generated} it is regular.
\end{proof}

\begin{cor}\label{cor:nomuleq1}
If $(X,L)$ is $K$-homogeneous, $Y\subset X\setminus L$ is a half-Maslov divisor and $u\colon(D,\partial D)\to (X,L)$ is a non-constant $J$-holomorphic disc then $u$ has Maslov index greater than or equal to 2.
\end{cor}
\begin{proof}
Any holomorphic disc has nonnegative Maslov index by positivity of intersections with the half-Maslov divisor. If $u$ is a holomorphic disc with Maslov index zero or one then by Lazzarini's theorem {\cite[Theorem A]{Laz}} there exists a somewhere injective disc, $u'$, with boundary on $L$ and Maslov index zero or one (Lazzarini provides a decomposition of $u$ into somewhere injective Riemann surfaces with boundary on $L$ and all of these contribute nonnegatively to the Maslov index by positivity of intersections with $Y$).

The moduli space $\mM^*_{0,1}(J,[u'])$ of somewhere injective $J$-discs in the same relative homology class as $u'$ having one boundary marked point is $(n-2)$-dimensional by Lemma \ref{lma-reg}. The evaluation map $\mM^*_{0,1}(J,[u'])\to L$ is $K$-equivariant, $L$ is $n$-dimensional and the $K$-action on $L$ is transitive. So if the moduli space is nonempty, it must have dimension at least $n$. This contradicts the fact that it is $(n-2)$-dimensional.
\end{proof}

Note that Lazzarini's theorem also implies that Maslov 2 discs are somewhere injective: otherwise one could extract a somewhere injective disc in the Lazzarini decomposition with strictly lower Maslov index, which cannot exist by Corollary \ref{cor:nomuleq1}. For this reason we will drop the star from the notation $\mM^*_{0,k}$ when dealing with Maslov 2 moduli spaces.

\subsection{Axial discs}

We are particularly interested in holomorphic discs which have extra symmetries. An axial disc is, roughly speaking, a disc with a one-parameter group of ambient isometries which preserve the disc setwise and rotate it about its centre.

\begin{dfn}
Suppose $(X,L)$ is $K$-homogeneous and $K_x$ is the stabiliser of $x\in L$. An {\em $x$-admissible homomorphism} is a homomorphism $R\colon\RR\to K$ such that $R(2\pi)\in K_x$. We say that $R$ is primitive if $R(\theta)\not\in K_x$ for all $\theta\in(0,2\pi)$.
\end{dfn}

\begin{dfn}
Let $R$ be an $x$-admissible homomorphism. A holomorphic disc $u\colon(D,\partial D)\to (X,L)$ with $u(1)=x$ is {\em $R$-axial} if (after a suitable reparametrisation) $u(e^{i\theta}z)=R(\theta)u(z)$ for all $z\in D$, $\theta\in\RR$. We say $u$ is {\em axial} without further qualification if there exists some reparametrisation and admissible homomorphism $R$ for which it is $R$-axial.
\end{dfn}

\begin{rmk}
An $R$-axial disc is simple if and only if $R$ is primitive.
\end{rmk}

\begin{lma}
Suppose that $(X,L)$ is $K$-homogeneous. Recall that the complexification $G$ of $K$ acts by holomorphic automorphisms on $X$. Given a point $x\in L$ and an $x$-admissible homomorphism $R\colon \RR\to K$ there is an $R$-axial disc $u_R\colon (D,\partial D)\to(X,L)$ with $u_R(1)=x$.
\end{lma}
\begin{proof}
Let $R^{\CC}\colon\CC\to G$ be the complexification of the admissible homomorphism (constructed by complexifying the Lie algebra homomorphism). The map
\[u(e^{a+i\theta})=R^{\CC}(a+i\theta)x,\qquad a\leq 0\]
defines an algebraic map $\CC^*\to X$ and the Zariski closure $\overline{u(\CC^*)}_Z$ of $u(\CC^*)$ is a rational curve containing $u(\CC^*)$ as a Zariski open subset (see {\cite[Proposition 15.2.1]{Taylor}}); thus $\overline{u(\CC^*)}_{Z}\setminus u(\CC^*)$ is a finite set of points and $u(\CC^*)$ is dense in $\overline{u(\CC^*)}_Z$ in the analytic topology. In particular, $u$ extends holomorphically over the punctures. The restriction of $u$ to the unit disc then gives a holomorphic disc $u_R$ with boundary on $L$.
\end{proof}

\subsection{Applications to Maslov 2 discs}

We will show that any Maslov 2 disc is axial.

\begin{lma}
If $(X,L)$ is $K$-homogeneous and $\beta$ is a relative homology class with Maslov number 2 then the evaluation map
\[\ev\colon\mM_{0,1}(J,\beta)\to L,\qquad\ev([u,z])=u(z)\]
is a local diffeomorphism when the moduli space is nonempty. The group $K$ acts transitively on components of $\mM_{0,1}(J,\beta)$.
\end{lma}
\begin{proof}
Since both spaces have dimension $n$, it suffices to show that the evaluation map has no critical points. The group $K$ acts on $\mM_{0,k}(J,\beta)$; an element $g\in K$ sends $[u,z]$ to $[gu,z]$. The evaluation map is $K$-equivariant and $L$ is a $K$-orbit. Therefore if $x\in L$ is critical, so is $gx$ for any $g\in K$. In particular all points in $L$ are critical, which contradicts Sard's theorem.

This shows that the $K$-orbit of $[u,z]$ is $n$-dimensional, connected and compact. It follows that this $K$-orbit is a component of the $n$-dimensional manifold $\mM_{0,1}(J,\beta)$.
\end{proof}

\begin{lma}\label{lma-rot-sym}
Suppose $(X,L)$ is $K$-homogeneous and write $K_x$ for the stabiliser of a point $x\in L$. Let $u$ be a Maslov 2 holomorphic disc in the class $\beta$. The $K$-stabiliser of $[u]\in\mM_{0,0}(J,\beta)$ has dimension $1+\dim K_x$ and $K$ acts transitively on components of $\mM_{0,0}(J,\beta)$.
\end{lma}
\begin{proof}
Since $\beta$ has Maslov number 2, the dimension of the moduli space $\mM_{0,k}(J,\beta)$ is $n+k-1$. We have seen that the evaluation map $\mM_{0,1}(J,\beta)\to L$ is a local diffeomorphism. Since the evaluation map is $K$-equivariant, the identity component of the $K$-stabiliser of $[u,z]\in\mM_{0,1}(J,\beta)$ is equal to the identity component of the $K$-stabiliser $K_x$ of $x=u(z)\in L$. The forgetful map $\mM_{0,1}(J,\beta)\to\mM_{0,0}(J,\beta)$ is $K$-equivariant and the fibre is one-dimensional. The stabiliser of $[u]$ is therefore of dimension $1+\dim K_x$.
\end{proof}

After making some further assumptions, we can classify all Maslov 2 discs with boundary on a $K$-homogeneous Lagr\-an\-g\-ian.

\begin{cor}\label{cor:axial-maslov2}
Suppose that $(X,L)$ is $K$-homogeneous. Suppose moreover that:
\begin{itemize}
\item the action of the complexification $G$ has a Zariski dense open orbit (this is necessarily the orbit containing $L$);
\item the complement of this open orbit is a half-Maslov divisor $Y\subset X$.
\end{itemize}
Then all Maslov 2 discs with boundary on $L$ are axial.
\end{cor}
\begin{proof}
Let $u\colon(D,\partial D)\to (X,L)$ be a Maslov 2 disc and suppose that $v\in\mathfrak{k}$ is a generator for the stabiliser subgroup of $[u]\in\mathcal{M}_{0,0}(J,\beta)$ guaranteed by Lemma \ref{lma-rot-sym}. Note that the sign of $v$ is determined by the requirement that it points along the boundary of $u$ oriented anticlockwise. Let $T>0$ be the smallest positive real number such that $\exp(Tv)\in K_x$ and let $R$ be the $u(1)$-admissible homomorphism $R(t)=\exp(tTv/2\pi)$. Then $t\mapsto R(t)u(1)$ is a parametrisation of the boundary of $u$ and is also the boundary of the $R$-axial disc $u_R$. This implies that the image of $u$ is contained in the $R$-axial holomorphic sphere $u_R\cup u_{R^{-1}}$. Moreover the image of $u$ must contain the image of $u_R$ since they share a common boundary and $u_R$ is embedded.

The Maslov number of a disc with boundary on $L$ is given by twice its intersection number with $Y$. We have assumed that the complement of the open orbit of $G$ is such an anticanonical divisor $Y$. By Corollary \ref{cor:nomuleq1}, any disc has Maslov number at least 2, so in particular the two axial discs $u_R$ and $u_{R^{-1}}$ intersect $Y$ nontrivially. Since $u$ has Maslov number 2, by positivity of intersections, it can intersect $Y$ at most once transversely at a smooth point. In particular its image in the rational curve $u_R\cup u_{R^{-1}}$ can only cover one hemisphere simply (or else it would intersect $Y$ in two or more points). Since the image of $u$ contains the image of $u_R$, this implies that $u_R=u\circ\phi$ for some reparametrisation $\phi$.
\end{proof}

\subsection{Applications to Maslov 4 discs}

We show that Maslov 4 discs which cleanly intersect a $K$-invariant complex submanifold of complex codimension 2 are necessarily axial.

\begin{cor}\label{cor-maslov4-symmetry}
Suppose that
\begin{itemize}
\item $(X,L)$ is $K$-homogeneous,
\item the complexification $G$ of $K$ has a Zariski open orbit whose complement is a half-Maslov divisor $Y$,
\item $\tilde{X}$ is a $K$-equivariant blow-up of $X$ along a $K$-invariant complex codimension 2 submanifold $Z\subset X$ disjoint from $L$. Let $u$ be a Maslov 4 holomorphic disc on $L$ such that $u(D)$ and $Z$ intersect cleanly in a single point.
\end{itemize}
Then $u$ is axial.
\end{cor}
\begin{proof}
The total transform of $Y$ is again a half-Maslov divisor. The proper transform of $u$ is a holomorphic disc $\tilde{u}$ in $(\tilde{X},L)$ which hits the exceptional divisor in a single point transversely. Therefore its Maslov index is $\mu(u)-2=2$. The $G$-action lifts to a holomorphic $G$-action on the blow-up and this still has an open orbit whose complement is the total transform of $Y$, which is still anticanonical. So $(\tilde{X},L)$ satisfies all the criteria of Corollary \ref{cor:axial-maslov2}. By Corollary \ref{cor:axial-maslov2}, therefore, $\tilde{u}$ is axial for some admissible homomorphism $R$, which implies that $u$ is also $R$-axial.
\end{proof}

Finally we will prove transversality for the two-point evaluation map from the moduli space of twice-marked Maslov 4 discs at points where the disc is axial.

\begin{lma}\label{lma:ev2-transversality}
Suppose that $(X,L)$ is $K$-homogeneous. Let $\beta$ be a relative homology class with Maslov 4. Suppose that for some admissible homomorphism $R\colon\RR\to K$ the disc $u$ is an embedded $R$-axial Maslov 4 holomorphic disc with boundary on $L$ representing the class $\beta$. Assume moreover that $T_{u(0)}X$ contains no vector which is fixed by the action of $S^1\cong R(\RR)\subset K$. Then for any $z_1,z_2\in\partial D$, $z_1\neq z_2$, the two-point evaluation map
\[\OP{\ev}\colon\mM^*_{0,2}(J,\beta)\to L\times L\]
is a submersion at $[u,z_1,z_2]$.
\end{lma}
\begin{proof}
Note that the Riemann-Hilbert pair $(u^*TX,u^*TL)$ contains a Maslov 2 Riem\-ann-Hil\-bert line subbundle $(TD^2,T\partial D^2)$ because $u$ is an embedding. We will use this fact in the proof. Also, since $u$ is an embedding, it is somewhere injective and hence the moduli space near $u$ is a manifold, so it makes sense to talk about the evaluation map being a submersion.

Decompose the Riemann-Hilbert pair $(u^*TX,u^*TL)$ into its Riemann-Hilbert summands $\bigoplus_{i=1}^3(E_i,F_i)$. We know from Lemma \ref{lma-reg} that the partial indices $\mu(E_i,F_i)$ are nonnegative. If the summands are ordered by increasing partial index then the possibilities are:

\begin{tabular}[htb]{cccc}
(a) & $0$ & $0$ & $4$\\
(b) & $0$ & $1$ & $3$\\
(c) & $0$ & $2$ & $2$\\
(d) & $1$ & $1$ & $2$
\end{tabular}

In each case the pair is filtered by
\[\mathcal{F}_k=\bigoplus_{\mu(E_i,F_i)\geq k}(E_i,F_i).\]
We claim that the action of $S^1=R(\RR)\subset K$ preserves this filtration. To see this, let $g\in S^1$ be a group element and suppose that $g\mathcal{F}_k\not\subset\mathcal{F}_k$. Then there is a section $\sigma$ of $\mathcal{F}_k$ such that $g\sigma$ projects nontrivially to a section $p_i(g\sigma)$ of a summand $(E_i,F_i)$ with $\mu(E_i,F_i)<k$. Being a section of $\mathcal{F}_k$, $\sigma$ has at least $k$ zeros (counted with multiplicity) and the same is therefore true of $p_i(g\sigma)$. But $p_i(g\sigma)$ is also a section of a one-dimensional Riemann-Hilbert pair with Maslov index strictly less than $k$, and therefore has strictly fewer than $k$ zeros (counted with multiplicity).

By a similar argument, the Maslov 2 subbundle $(TD^2,T\partial D^2)$ sits inside $\mathcal{F}_2$. This precludes cases (a) and (b) since $\mathcal{F}_2$ is then a Riemann-Hilbert summand of the wrong Maslov index. In case (d) this means that the normal bundle has partial indices $\mu(E_1,F_1)=\mu(E_2,F_2)=1$ and the result follows from Lemma \ref{lma-transverse-twopoint}.

It remains to rule out case (c). The space of holomorphic sections of $(u^*TX,u^*TL)$ is a representation of $S^1$ and contains the subspace of sections $H^0(\mathcal{F}_2)$ as a subrepresentation. The complement of $H^0(\mathcal{F}_2)$ is (real) one-dimensional and is therefore a trivial subrepresentation. This implies that there exists a vector in $T_{u(0)}X$ (the section evaluated at $z=0$) which is fixed by the $S^1$-action. This contradicts the assumption on $T_{u(0)}X$.
\end{proof}

\part{An example: the Chiang Lagrangian}

\section{Quasihomogeneous threefolds of $\SL$}

The Chiang Lagrangian is the first in a family of examples of homogeneous Lagrangians. We describe these in greater generality in this section because it seems most natural. We believe that our methods should generalise to the higher examples. From Section \ref{sect:topol} we will specialise to the case of the Chiang Lagrangian.

\subsection{$\SL$-orbits in $\Sym^n\cp{1}$}

Let $V$ denote the standard two-dimensional complex representation of $\SL$. The varieties
\[\cp{n}=\PP(\Sym^nV)\cong\Sym^n\PP(V)\]
are isomorphic as $\SL$-spaces, so we can think of a configuration of $n$ points on $\cp{1}=\PP(V)$ as a point in the projective space $\cp{n}$. Let $C$ be a configuration of $n\geq 3$ distinct points on $\cp{1}$ and consider the closure $\overline{\SL\cdot C}$ of its $\SL$-orbit in $\cp{n}$. This is a {\em quasihomogeneous} complex threefold $X_C$, in other words there is a dense Zariski-open $\SL$-orbit.

There are precisely four cases in which $X_C$ is smooth \cite{AluFab}; we will specify these by giving a representative configuration from the orbit:
\begin{itemize}
\item $C=\Delta$, the set of zeros of the polynomial $x(x^2+3y^2)=0$ in $\cp{1}$. In $S^2\cong\cp{1}$ these zeros lie at the vertices
\[(0,0,1),\qquad (0,\sqrt{3}/2,-1/2),\qquad (0,-\sqrt{3}/2,-1/2)\]
of an equilateral triangle. The stabiliser $\Gamma_{\Delta}$ of $\Delta$ is the binary dihedral group $\tilde{D}_3$ of order twelve.
\item $C=T$, the vertex set of a regular tetrahedron on $\cp{1}$; equivalently the zeros of the polynomial $x^4+2i \sqrt{3} x^2 y^2 + y^4$. The stabiliser $\Gamma_T$ of $T$ is the binary tetrahedral group (order 24).
\item $C=O$, the vertex set of a regular octahedron on $\cp{1}$; equivalently the zeros of the polynomial $xy(x^4-y^4)$. The stabiliser $\Gamma_O$ of $O$ is the binary octahedral group (order 48).
\item $C=I$, the vertex set of a regular icosahedron on $\cp{1}$; equivalently the zeros of the polynomial $xy(x^{10}+11x^5y^5-y^{10})$. The stabiliser $\Gamma_I$ of $I$ is the binary icosahedral group (order 120).
\end{itemize}
The corresponding varieties $X_C$ have $b_2=1$, $b_3=0$ and are Fano. The first Chern class of $X_C$ is $c_1(X_C)=\ell_CH$ where $\ell_C=4,3,2,1$ for $C=\Delta,T,O,I$. The cohomology ring is
\[H^*(X_C;\ZZ)=\ZZ[H,E]/(H^2=k_CE,\ E^2=0)\]
where $k_{\Delta}=1$, $k_{T}=2$, $k_O=5$, $k_I=22$. In fact $X_{\Delta}\cong\cp{3}$, $X_T$ is a quadric threefold, $X_O$ is the Del Pezzo threefold $V_5$ of degree five and $X_I$ is the Mukai-Umemura threefold $V_{22}$. Note that the Coxeter-Dynkin diagrams of the finite stabiliser groups are the $E_5$, $E_6$, $E_7$ and $E_8$ diagrams.

\subsection{Geometry of the compactification}

Each of these varieties has a decomposition as
\[X_C=W_C\cup Y_C\]
where $W_C$ is the open orbit $\SL\cdot C$ which is isomorphic to $\SL/\Gamma_C$ and $Y_C$ is a compactification divisor preserved by the $\SL$-action (see {\cite[Lemma 1.5]{Muk-Ume}}).

The divisor $Y_C$ consists of all $n$-tuples of points where $n-1$ of the points coincide. Inside $Y_C$ is the locus $N_C$ consisting of all $n$-tuples of coincident points. In another language, $N_C$ is the rational normal curve coming from the canonical embedding $\PP(V)\to\PP(\Sym^nV)$ and $Y_C$ is its tangent variety.

The orbit decomposition of $X_C$ is therefore $W_C\cup (Y_C\setminus N_C)\cup N_C$. The singular divisor $Y_C$ is anticanonical in each case.

Inside each of the open orbits $W_C$ is a copy of $L_C=SU(2)/\Gamma_C$, the $SU(2)$-orbit of $C$. This is {\em a priori} a totally real submanifold; we will see that for a suitable choice of K\"{a}hler form on $X_C$ it is a Lagrangian submanifold.

\subsection{K\"{a}hler form and moment map}
\label{moment}

Let $x$ and $y$ be coordinates on $V^*$; consider $\Sym^nV$ as the space of polynomials $p(x,y) = \sum_{k=0}^n v_kx^{k}y^{n-k}$ in $x$ and $y$ and use the coefficients $v_k$ as homogeneous coordinates on $\PP(\Sym^nV)$. Recall that the $(n+1)$-dimensional irreducible representation of $SU(2)$ is defined by:
\[ \left(\begin{array}{cc}
\alpha & \beta\\
-\bar{\beta} & \bar{\alpha}  
\end{array}\right) \cdot p(x,y) \mapsto p(\alpha x - \bar{\beta} y , \beta x + \bar{\alpha} y ) \]
The representation $\Sym^nV$ inherits a Hermitian inner product from the standard Hermitian inner product on $V^*$, for which $|x|^2=|y|^2=1$ and $x\cdot y=0$. On $\Sym^nV$ with respect to the coordinates $v_k$ this gives us an invariant K\"{a}hler form (cf. \cite{bombieri}):
\[\frac{i}{2}\sum_{k=0}^n\binom{n}{k}^{-1}dv_k\wedge d\overline{v}_k = \frac{i}{2}\sum_{k=0}^n du_k\wedge d\overline{u}_k\]
where we introduced unitary coordinates $u_k = \binom{n}{k}^{-1/2} v_k$ for convenience. The action of $SU(2)$ on $\Sym^n V$ commutes with the diagonal action of $S^1$ given by \[ \theta : p(x,y) \mapsto p(e^{i \theta} x , e^{i\theta} y) \] This also preserves the above K\"ahler form, hence via symplectic reduction with respect to the diagonal $S^1$ action, we get a Hamiltonian $SU(2)$ action on the projective space $\cp{n} = \PP(\Sym^n V)$ equipped with the standard Fubini-Study form. 

Now, $X_C$ is a projective variety sitting inside $\cp{n}$ which consists of a union of $SU(2)$-orbits. By restriction, we induce a symplectic structure and a Hamiltonian $SU(2)$-action on $X_C$. The equivariant moment map is given in coordinates on $\Sym^nV$ by
{\footnotesize \[\mu_n(u_0,\ldots,u_n)= \left(\begin{array}{cc}
 i \sum_{k=0}^n(n-2k)|u_k|^2 & 2i \sum_{k=0}^{n-1} \sqrt{ (k+1) (n-k)} u_k \overline{u}_{k+1} \\
 2i \sum_{k=0}^{n-1} \sqrt{ (k+1) (n-k)} \overline{u}_{k} u_{k+1} & -i \sum_{k=0}^n(n-2k)|u_k|^2
\end{array}\right) \]}
\par\noindent where we identified the Lie algebra $\mathfrak{su}(2)$ with its dual $\mathfrak{su}(2)^*$ via the invariant bilinear form $\langle A, B \rangle = \frac{1}{4} \text{Tr} (AB) $.

We can now check that in each case, $L_C=\mu_n^{-1}(0)$ so that $L_C$ is a Lagrangian submanifold. It suffices to check that the point in $\cp{n}$ corresponding to $C$ is in $\mu_n^{-1}(0)$. This is easy to check since the configurations $C$ are given in homogeneous coordinates on $\cp{n}$ by
\begin{align*}
\Delta &=[1:0:3:0]\\
T&=[1:0:2i \sqrt{3}:0:1]\\
O&=[0:1:0:0:0:1:0]\\
I&=[0:1:0:0:0:0:11:0:0:0:0:-1:0].
\end{align*}
\begin{rmk}
In the case $\Delta$, this is the Chiang Lagrangian \cite{Chiang}.
\end{rmk}
The first homology of $L_C$ is $H_1(L_C;\ZZ)=\ZZ/\ell_C$. In each case the long exact sequence in relative homology gives
\[0\to \ZZ\to H_2(X_C,L_C;\ZZ)\to\ZZ/\ell_C\to 0\]
and it is easy to find discs whose boundaries generate $H_1(L_C;\ZZ)$ so $H_2(X_C,L_C;\ZZ)=\ZZ$ (see Example \ref{exm:maslov2} for such discs in the case $C=\Delta$).
\begin{rmk}
For more general compact Lie groups $K$ there is the following result of Bedulli and Gori {\cite[Theorem 1]{BG}}. Let $K$ be a compact Lie group of dimension $n$ and let $G$ denote its complexification. Let $X$ be a (real) $2n$-dimensional compact K\"{a}hler manifold with $h^{1,1}=1$ admitting a Hamiltonian action of $K$ by K\"{a}hler isometries. The $K$-action complexifies to an action of $G$ and we will further assume that this complexified action has a dense Zariski-open orbit whose stabiliser is a finite group $\Gamma\subset K$. Then there is an equivariant moment map $\phi\colon X\to\mathfrak{k}^*$ and the fibre over zero is a Lagrangian $K$-orbit diffeomorphic to $K/\Gamma$.
\end{rmk}
\begin{rmk}
In fact there is a complete classification of $\SL$-equivariant compactifications of $\SL/\Gamma$ for a finite subgroup $\Gamma$ due to Nakano \cite{Nak}, building on work of Mukai and Umemura \cite{Muk-Ume}. There are two further examples with $b_2=1$, namely the standard actions of $\SL$ on $\cp{3}$ and the quadric threefold, where the corresponding Lagrangians are respectively the standard $\RR\PP^3$ and real ellipsoid. These have minimal Maslov numbers 4 and 6 respectively, in contrast to the examples $X_C$ which have minimal Maslov 2 (see Lemma \ref{lma:monotone} below).
\end{rmk}

\subsection{Chern and Maslov classes}\label{sct-chern-maslov}

Note that $Y_C\subset X_C$ is an anticanonical divisor: the $\SL$-action on $X_C$ defines a bundle map $\alpha\colon X_C\times\mathfrak{sl}(2,\CC)\to TX_C$ which is an isomorphism on the open orbit $X_C\setminus Y_C$; along $Y_C$ the holomorphic $n$-form $\Lambda^3\alpha$ vanishes so $Y_C$ is anticanonical, see {\cite[Section 3]{hitchin}}. In particular, the Chern class evaluated on a holomorphic curve is equal to the homological intersection number of the curve with $Y_C$.

The Lagrangian $L_C$ is disjoint from $Y_C$ and has constant phase for the volume form $\Lambda^3\alpha$; hence the Maslov class of a holomorphic disc with boundary on $L_C$ is equal to its (relative) homological intersection number with $Y_C$, see {\cite[Lemma 3.1]{Aur}}. In the language of Definition \ref{dfn:half-maslov} $Y_C$ is a half-Maslov divisor.

\begin{lma}\label{lma:monotone}
The Lagrangians $L_C$ are monotone
\[\omega(\beta)=K\mu(\beta)\]
for some $K>0$, and have minimal Maslov number $2$.
\end{lma}
\begin{proof}
Let $P$ be an axis through a point $v\in C$ and its antipode $\bar{v}$ and let $S^1\subset SU(2)$ be the subgroup of rotations fixing this axis. If $\CC^*\subset\SL$ denotes the complexification of $S^1$ then the closure of the $\CC^*$-orbit of $C$ is a holomorphic sphere which intersects $Y_C$ twice: once transversely at a smooth point (the configuration comprising $v$ and the $(n-1)$-fold point at $\bar{v}$) and once elsewhere\footnote{If $\bar{v}$ is in $C$ then the sphere intersects $Y_C$ at the smooth point comprising $\bar{v}$ and the $(n-1)$-fold point at $v$; otherwise it intersects at the $n$-fold point at $v$.}. The hemisphere containing the transverse intersection at a smooth point is therefore a Maslov 2 holomorphic disc with positive area. Since $H_2(X_C,L_C;\ZZ)\cong\ZZ$ this is enough to prove the lemma.
\end{proof}

\subsection{Quantum cohomology and eigenvalues of the first Chern class}

The quantum cohomology of $X_C$ is computed in {\cite[Section 2]{BayMan}}. We consider it as a $\ZZ/2$-graded ring (in particular we set the Novikov variable $q=1$). It is
\[QH^*(X_C)=\ZZ[H,E]/(H^2=k_CE+R_C,\ E^2=Q_C)\]
where $H$ and $E$ have grading zero and the quantum contributions $R_C$ and $Q_C$ are given in Figure \ref{fig:c1-char-poly}.

\begin{figure}[htb]
\begin{center}
\begin{tabular}{|c|c|c|}
\hline
$C$ & $R_C$ & $Q_C$\\ \hline
$\Delta$ & $0$ & $1$\\
$T$ & $0$ & $H$ \\
$O$ & $3$ & $E+1$\\
$I$ & $2H+24$ & $2E+H+4$\\ \hline
\end{tabular}
\caption{The quantum contributions to the relations in the $QH^*(X_C)$.}
\label{fig:q-contrib}
\end{center}
\end{figure}

The eigenvalues of $c_1\star\colon QH^*(X_C)\to QH^*(X_C)$ (over a field $\kk$) are important for Lagrangian Floer theory. They arise as counts of Maslov 2 discs with boundary on monotone Lagrangian submanifolds whose Floer cohomology over $\kk$ is non-vanishing \cite{Aur}. More precisely:
\begin{dfn}
If $L\subset X$ is a monotone Lagrangian submanifold then the invariant $\mathfrak{m}_0(L)$ is defined to be the sum (over relative homology classes $\beta\in H_2(X,L;\ZZ)$ with Maslov number 2) of degrees of evaluation maps $\mM_{0,1}(J,\beta)\to L$ where $J$ is a regular compatible almost complex structure and $\mM_{0,1}(J,\beta)$ denotes the moduli space of $J$-holomorphic discs representing the class $\beta$ with one marked point on the boundary.
\end{dfn}
\begin{prp}[Auroux \cite{Aur}, Kontsevich, Seidel]\label{prp:c1-eval}
Let $\kk$ be a field of characteristic not equal to 2. If $L\subset X$ is a compact, orientable, spin, monotone Lagrangian submanifold whose self-Floer cohomology over $\kk$ is non-vanishing then $\mathfrak{m}_0(L)$ is an eigenvalue of $c_1\star$ acting on $QH^*(X;\kk)$.
\end{prp}
The characteristic polynomial $\chi_C(\lambda)$ of the matrix $c_1\star$ in each of our examples can be calculated by hand using the presentation above. We list these characteristic polynomials in Figure \ref{fig:c1-char-poly}.

\begin{figure}[htb]
\begin{center}
\begin{tabular}{|c|c|c|}
\hline
$C$ & $\chi_C(\lambda)$\\ \hline
$\Delta$ & $\lambda^4-256$\\
$T$ &  $\lambda(\lambda^3-108)$\\
$O$ & $\lambda^4-44\lambda-16$\\
$I$ & $\lambda^4-4\lambda^3-88\lambda^2-300\lambda-304$\\ \hline
\end{tabular}
\caption{The characteristic polynomial of quantum multiplication by the first Chern class for the quasihomogeneous varieties $X_C$.}
\label{fig:c1-char-poly}
\end{center}
\end{figure}

\section{The topology of the Chiang Lagrangian, $L_{\Delta}$}\label{sect:topol}

We now specialise to the case $C=\Delta$, the equilateral triangle with vertices at
\[(0,0,1),\qquad (0,\sqrt{3}/2,-1/2),\qquad (0,-\sqrt{3}/2,-1/2)\]
on $S^2\cong\cp{1}$. We obtain a Lagrangian $L_{\Delta}=SU(2)/\Gamma_{\Delta}\subset\cp{3}=X_{\Delta}$.

\subsection{A fundamental domain}

 The stabiliser of $C$ is the {\em binary dihedral group of order twelve}:
\[\Gamma_{\Delta}=\left\{\left(\begin{array}{cc}\omega & 0 \\ 0 & \bar{\omega}\end{array}\right)\ :\ \omega^6=1\right\}\cup\left\{\left(\begin{array}{cc}0 & i\bar{\omega} \\ i\omega & 0\end{array}\right)\ :\ \omega^6=1\right\}\subset \SU(2)\subset\SL\]
Note that the action of $\SU(2)$ on $\cp{1}$ is the usual quaternionic rotation action: if $\mathbf{u}=(u_1,u_2,u_3)$ is a unit vector and
\[\sigma_1=\left(\begin{array}{cc}i & 0 \\ 0 & -i\end{array}\right),\ \sigma_2=\left(\begin{array}{cc}0 & 1 \\ -1 & 0\end{array}\right),\ \sigma_3=\left(\begin{array}{cc}0 & i \\ i & 0\end{array}\right)\]
are the Pauli matrices then $\exp(\theta\sum u_i\sigma_i)$ acts as a right-handed rotation by $2\theta$ around the axis $\mathbf{u}$.

We will identify $g\in SU(2)/\Gamma_{\Delta}$ with the point $g\Delta\in L_{\Delta}$. The universal cover $SU(2)\cong\tilde{L}_{\Delta}$ is tiled by twelve fundamental domains related by the action of $\Gamma_{\Delta}$: each domain is a hexagonal prism centred at the corresponding element of $\Gamma_{\Delta}$. This tiling shown in Figure \ref{fig:prism}, stereographically projected so that the identity sits at the origin.

A single fundamental domain comes with face and edge identifications such that the quotient space is $L_{\Delta}$. Opposite quadrilateral faces are identified by a right-handed twist of $\pi/2$ radians and opposite hexagonal faces are identified by a right-handed twist of $\pi/3$ radians; the corresponding edge identifications are indicated in Figure \ref{fig:edge-identifications}. The resulting cell structure on $L_{\Delta}$ has three vertices (denoted $x_1$, $x_2$ and $x_3$ in the figure), six 1-cells (denoted 1, 2, 3, 4, 5, 6 in the figure), four 2-cells and a 3-cell. If $X$ is a subset of the fundamental domain, we will write $\bar{X}$ for the corresponding subset of the quotient $L_{\Delta}$.

\begin{figure}[htb]
\begin{center}
\includegraphics[width=300px]{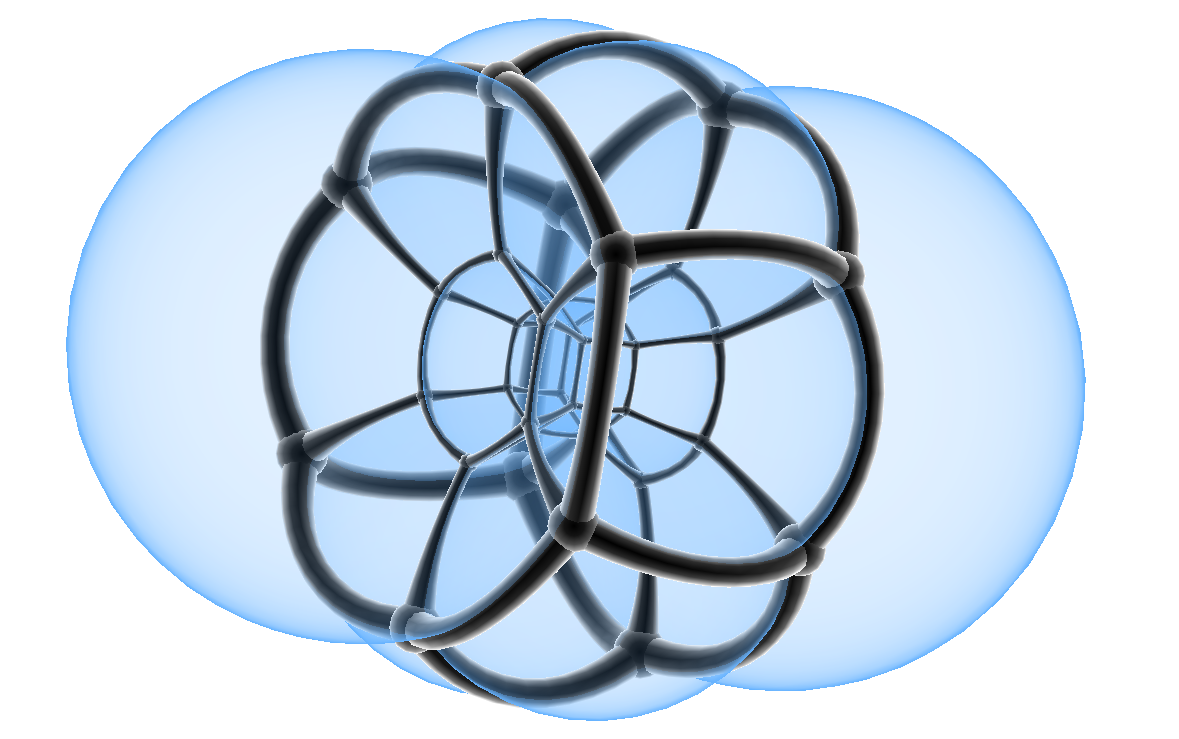}
\caption{The tiling of $SU(2)$ by fundamental domains for $\Gamma_{\Delta}$, the binary dihedral group of order twelve. Picture produced using \href{http://www.jenn3d.org}{Jenn3d}.}
\label{fig:prism}
\end{center}
\end{figure}

\begin{figure}[htb]
\begin{center}
\includegraphics{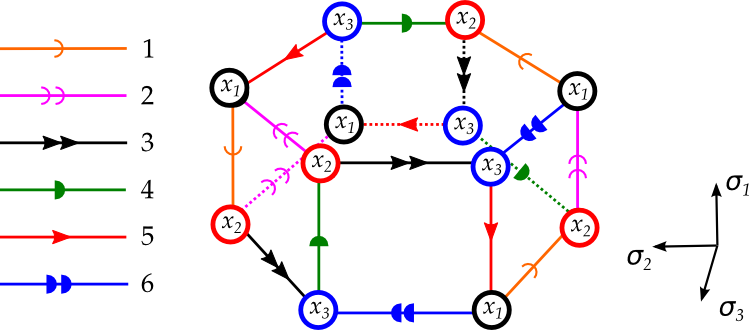}
\caption{Edge identifications for the fundamental domain. For orientation we also include the axes corresponding to the Pauli matrices $\sigma_1,\sigma_2,\sigma_3\in\mathfrak{su}(2)=T_1SU(2)$.}
\label{fig:edge-identifications}
\end{center}
\end{figure}

\subsection{A Heegaard splitting}

Take the union $S$ of all the 1-cells and the two hexagonal faces; in the quotient $L_{\Delta}$ this descends to a hexagon $\bar{S}$ with opposite vertices identified. Indeed $\bar{S}$ retracts onto $\bar{T}$ where $T$ is the union of a two slightly smaller hexagons, each with six radial prongs connecting it to the vertices (see Figure \ref{fig:spider}). An open neighbourhood $N$ of $\bar{S}$ (or $\bar{T}$) is therefore a genus 3 handlebody. Note that the complement $N'=L_{\Delta}\setminus N$ is also a genus 3 handlebody which retracts onto the wedge of three circles $\bar{\alpha}_1'\cup\bar{\alpha}_2'\cup\bar{\alpha}_2'$ where $\alpha_1',\alpha_2',\alpha_3'$ are the three axes of the prism through the centre, $m'$, passing through the midpoints $x_1'$, $x_2'$ and $x_3'$ of the quadrilateral faces. The decomposition $L_{\Delta}=N\cup N'$ is therefore a Heegaard splitting (see Figure \ref{fig:heegaard}).

\begin{figure}[htb]
\begin{center}
\includegraphics{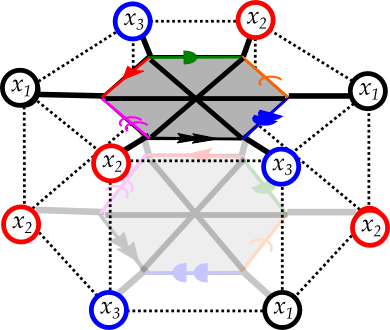}
\caption{The subset $T$ comprising two hexagons (one faint on the bottom face) each with six prongs connecting them to the vertices.}
\label{fig:spider}
\end{center}
\end{figure}

\begin{figure}[htb]
\begin{center}
\includegraphics{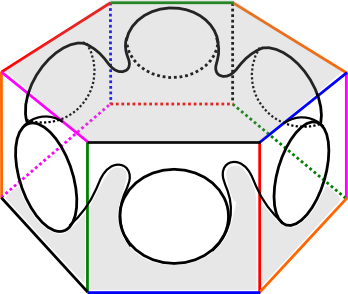}
\caption{The genus 3 Heegaard splitting.}
\label{fig:heegaard}
\end{center}
\end{figure}

\begin{figure}[htb]
\begin{center}
\includegraphics{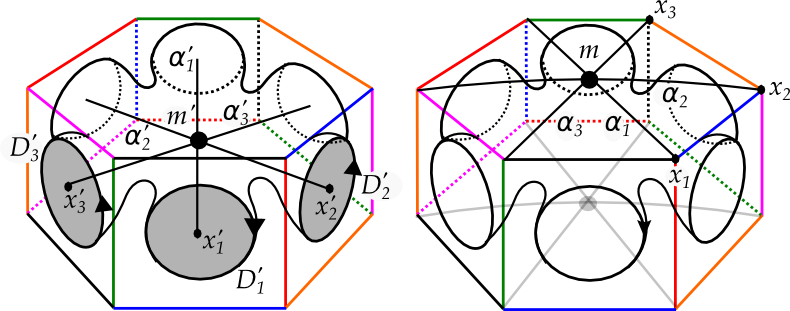}
\caption{The Heegaard splitting is associated with a Morse function. Here we see: (left) the minimum ($m'$), the index one critical points ($x_1'$, $x_2'$, and $x_3'$), their ascending discs ($D_1'$, $D_2'$ and $D_3'$) in grey and descending manifolds ($\alpha_1'$, $\alpha_2'$ and $\alpha_3'$); (right) the maximum ($m$), the index two critical points ($x_1$, $x_2$ and $x_3$) and their ascending manifolds ($\alpha_1$, $\alpha_2$ and $\alpha_3$).}
\label{fig:morse}
\end{center}
\end{figure}

\subsection{A Morse function}

This Heegaard splitting comes from a Morse function with a minimum at the centre, $m'$, a maximum at the midpoint, $m$, of the hexagonal faces, three index one critical points at the midpoints $x_1'$, $x_2'$ and $x_3'$ of the quadrilateral faces and three index two critical points at the vertices $x_1$, $x_2$, and $x_3$.

\begin{itemize}
\item In $N'$, the ascending manifolds of $x_1'$, $x_2'$ and $x_3'$ are the discs of intersection between $N'$ and the quadrilateral faces. The descending manifolds of $x_1'$, $x_2'$ and $x_3'$ are $\alpha_1'$, $\alpha_2'$ and $\alpha_3'$.
\item Figure \ref{fig:N-handlebody} shows the handlebody $N$ as a neighbourhood of $\bar{T}$. The thick lines are the flowlines $\alpha_1$, $\alpha_2$ and $\alpha_3$ connecting $x_1$, $x_2$, and $x_3$ to the maximum. The smaller shaded discs are the descending manifolds of $x_1$, $x_2$, and $x_3$.
\end{itemize}

Consider the 3-Sylow subgroup $\ZZ/3\subset\Gamma_{\Delta}$ which rotates the hexagonal prism through multiples of $2\pi/3$. Note that our Morse function can be chosen to be invariant under the left action of $\ZZ/3$ on $SU(2)/\Gamma_{\Delta}$.

\begin{figure}[htb]
\begin{center}
\includegraphics{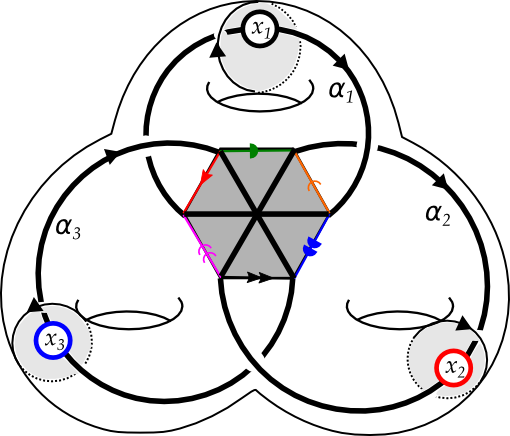}
\caption{The handlebody $N$ as a neighbourhood of $\bar{T}$. The thick lines are the flowlines connecting $x_1$, $x_2$, and $x_3$ to the maximum. The smaller shaded discs are the descending manifolds of $x_1$, $x_2$, and $x_3$.}
\label{fig:N-handlebody}
\end{center}
\end{figure}

From the edge identifications we can read off how the boundaries of the ascending manifolds $D_1'$, $D_2'$ and $D_3'$ of $x_1'$, $x_2'$ and $x_3'$ intersect the descending manifolds of $x_1$, $x_2$ and $x_3$ and hence compute the Morse differential. Consider the loop $\gamma_1'$ of intersection between $D_1'$ and the Heegaard surface. Pushing this loop into the handlebody $N$ it represents the element
\[\gamma_1'=\alpha_3\alpha_1\alpha_3\alpha_2\in\pi_1(N).\]
as can be seen in Figure \ref{fig:pi1}.

\begin{figure}[htb]
\begin{center}
\includegraphics{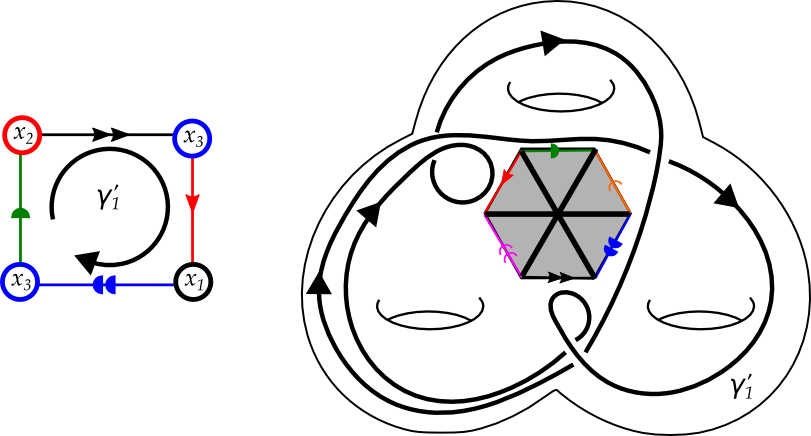}
\caption{The loop $\gamma_1'=\partial D_1'$ as an element of $\pi_1(N)$.}
\label{fig:pi1}
\end{center}
\end{figure}

Similarly
\[\gamma_2'=\alpha_1\alpha_2\alpha_1\alpha_3,\qquad\gamma_3'=\alpha_2\alpha_3\alpha_2\alpha_1.\]
If we assign orientations to the ascending and descending manifolds as indicated by the arrows in Figures \ref{fig:morse} and \ref{fig:N-handlebody} then the intersection number of the loop $\alpha_i$ and the descending manifold $D_j$ of the index two critical point $x_j$ is $\delta_{ij}$. Therefore the only non-vanishing Morse differentials are
\begin{align}
\nonumber dx_1'&=x_1+x_2+2x_3\\
\nonumber dx_2'&=2x_1+x_2+x_2\\
\label{eq:morse-diff}dx_3'&=x_1+2x_2+x_3.
\end{align}
Recall that the Morse function is invariant under the left action of $\ZZ/3\subset\Gamma_{\Delta}$ on $SU(2)/\Gamma_{\Delta}$. The choice of orientations is also symmetric. This action cyclically permutes $x_1,x_2,x_3$ which accounts for the cyclic symmetry of the Morse complex.

\section{Holomorphic discs on $L_{\Delta}$}

We now proceed to find all the $J$-holomorphic discs with boundary on $L_{\Delta}$ we need for the calculation of Floer cohomology, where $J$ is the $SU(2)$-invariant K\"{a}hler complex structure.

\subsection{Axial discs}

The possible primitive $\Delta$-admissible homomorphisms $R\colon\RR\to SU(2)$ fall into three classes according to whether the order of $R(2\pi)$ is $2$, $4$ or $6$. Order $4$ will yield axial Maslov 2 discs; order $6$ will yield axial Maslov 4 discs.

\begin{exm}[Axial Maslov 2 discs]\label{exm:maslov2}
Consider the homomorphism $R_1(e^{i\theta})=\exp(\theta\sigma_3/4)$. This acts on the triangle $\Delta\in L_{\Delta}$ by rotating it through an angle $\theta/2$ around the $z$-axis. After an angle $\theta=2\pi$ the triangle $\Delta$ has moved around a loop in $L_{\Delta}$ representing a generator of $H_1(L_{\Delta};\ZZ)$, swapping the two vertices $(0,\pm\sqrt{3}/2,-1/2)$. This loop bounds the axial holomorphic disc $u_{R_1}$ represented by the shaded area in Figure \ref{fig:chiang-disc2}. There are three discs like this passing through $\Delta$, corresponding to the $\Delta$-admissible homomorphisms
\[R_1(\theta)=\exp(\theta\sigma_3/4),\ R_2(\theta)=\exp(\theta(\sigma_2\sqrt{3}-\sigma_3)/8),\ R_3(\theta)=\exp(\theta(\pm\sigma_2\sqrt{3}-\sigma_3)/8)\]
around the axes through the three vertices of $\Delta$. It follows from Corollary \ref{cor:axial-maslov2} that these are all of the Maslov 2 discs through $\Delta$.
\end{exm}
Similarly we see that
\begin{lma}\label{class2}There are precisely three Maslov 2 discs through any point $g\Delta$, corresponding to the $g\Delta$-admissible homomorphisms $gR_ig^{-1}$, $i=1,2,3$.\end{lma}

\begin{figure}
\begin{center}
\includegraphics[width=130px]{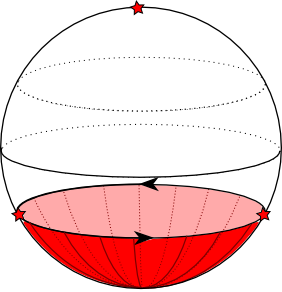}
\caption{A Maslov 2 disc $u\colon (D,\partial D)\to (\cp{3},L_{\Delta})$ passing through $\Delta$. The two vertices $(0,\pm\sqrt{3}/2,-1/2)$ move toward the south pole as $z\in D$ moves toward the origin and they rotate along the arrows as $z\in\partial D$ moves around the boundary.}
\label{fig:chiang-disc2}
\end{center}
\end{figure}

\begin{exm}[Axial Maslov 4 discs]\label{exm:maslov4}
Consider the homomorphism $R(\theta)=\exp(\theta\sigma_1/6)$. This acts on the triangle $\Delta\in L_{\Delta}$ by rotating it through an angle $\theta/3$ around the $x$-axis. After an angle $\theta=2\pi$ the triangle $\Delta$ has moved around a loop in $L_{\Delta}$ representing the element of order two in $H_1(L_{\Delta};\ZZ)$, cyclically permuting the three vertices of $\Delta$. This loop bounds the axial holomorphic disc $u_R$ represented by the shaded area in Figure \ref{fig:chiang-disc4} (drawn after rotating to make the $x$-axis vertical for clarity). There are two discs like this passing through $\Delta$, the other corresponding to $R(\theta)=\exp(-\theta\sigma_1/6)$.
\end{exm}

\begin{figure}
\begin{center}
\includegraphics[width=130px]{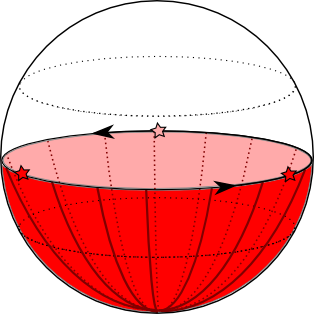}
\caption{A Maslov 4 disc $u\colon (D,\partial D)\to (\cp{3},L_{\Delta})$ passing through $\Delta$. Note that for clarity this is drawn after a rotation to make the $x$-axis vertical. The three vertices of $\Delta$, denoted by stars, are allowed to rotate in the direction of the arrows around the vertical axis through $2\pi/3$, tracing out a loop in $L_{\Delta}$. This loop bounds a holomorphic disc: as we move towards the centre of the disc the triple of points move together towards the south pole, tracing out the shaded region.}
\label{fig:chiang-disc4}
\end{center}
\end{figure}

\subsection{Maslov 4 discs through $m'$ and $m$}

Recall that $m'=\Delta$ and $m=\exp(\pi\sigma_1/6)\Delta$ are the minimum and maximum respectively of our Morse function. The count of Maslov 4 discs passing through these two points will be crucial in determining the Floer differential. The aim of this section is to prove:

\begin{prp}
There are two Maslov 4 discs with boundary on $L_{\Delta}$ passing through both $m'$ and $m$. They are precisely the axial discs constructed in Example \ref{exm:maslov4}.
\end{prp}

It is enough to prove that the Maslov 4 discs through both $m'$ and $m$ are axial.

The intersection pattern of a Maslov 4 disc with the divisor $Y_{\Delta}$ is one of the three following possibilities: the disc intersects $N_{\Delta}$ cleanly in a single point; the disc intersects $Y_{\Delta} \setminus N_{\Delta}$ transversely in two points; the disc intersects $Y_{\Delta} \setminus N_{\Delta}$ tangentially at one point (with multiplicity 2). This follows from positivity of intersections and the fact that the Maslov number is determined by the relative homological intersection of the disc with the anticanonical divisor $Y_{\Delta}$ which has a cuspidal singularity along $N_{\Delta}$. If the disc intersects $N_{\Delta}$ cleanly in a single point then it is axial (Corollary \ref{cor-maslov4-symmetry}). Our task is to rule out the other possibilities. We will do this by projecting to a lower-dimensional problem. This argument was inspired by Hitchin's paper \cite{hitchin}.

\subsubsection{Constructing a projection}

The one-parameter subgroups of $SU(2)$ act as rotations around a fixed axis, which is the same as a pair of antipodal points in $\cp{1}$. Given an axis $P$, its stabiliser in $\SL$ is $\pin_P$, isomorphic to the group $\pin$ whose definition we briefly recall. The group $O(2)$ has two double covers (central extensions)
\[1\to\ZZ/2\to{Pin}_{\pm}(2)\to O(2)\to 1\]
corresponding to whether the preimage of a reflection squares to the identity or to the nontrivial central element. If we fix a pair $P$ of antipodal points on $\cp{1}$ then their stabiliser in $SO(3)$ is a copy of $O(2)$ which we denote by $O(2)_P$. For instance, if $P=\{0,\infty\}$ then the reflections preserving $P$ are given by the matrices
\[\left(\begin{array}{ccc}
\cos\theta & \sin\theta & 0 \\
\sin\theta & -\cos\theta & 0\\
0 & 0 & -1
\end{array}\right)\in SO(3)\]
The preimage of such a matrix in ${Spin}(3)$ squares to the nontrivial central element in ${Spin}(3)$, so the preimage of $O(2)_P$ is isomorphic to ${Pin}_-(2)$ and written $Pin_-(2)_P$. We will write $Pin_-(2,\CC)_P\subset\SL$ for the complexification of ${Pin}_-(2)_P\subset SU(2)={Spin}(3)$. Note that if $P=\{0,\infty\}$, then
\[Pin_-(2,\CC)_P=\left\{ \left(\begin{array}{cc}
u & 0 \\
0 & u^{-1} 
\end{array}\right)\ :\ u\in\CC^*\right\} \cup 
\left\{\left(\begin{array}{cc}
0 & v \\
-v^{-1} & 0 
\end{array}\right)\ :\ v\in\CC^*\right\}. \] 

For convenience we will rotate so that $\Delta$ consists of the third roots of unity, which sit in $S^2$ in a plane orthogonal to the axis through $P_{\Delta}=\{0,\infty\}$. The group $\Gamma_{\Delta}$ is then contained in $\pin_{P_{\Delta}}$ so there is a map
\[\SL/\Gamma_{\Delta}\to\SL/\pin_{P_{\Delta}}.\]
This yields a rational map $\cp{3}\to\cp{2}$ and it extends to a dominant regular map $p\colon\widetilde{\cp{3}}\to\cp{2}$ from the blow-up of $\cp{3}$ along the twisted cubic curve $N_{\Delta}$. This map can be understood as follows.

It is well-known that through every point of $\cp{3}\setminus N_{\Delta}$ there is a unique secant or tangent line of $N_{\Delta}$, see for example {\cite[Chapter XII, Theorem 2]{SempleKneebone}}. This line intersects $N_{\Delta}$ in two points (counted with multiplicity) so we get a map $\cp{3}\setminus N_{\Delta}\to\OP{Sym}^2\left(N_{\Delta}\right)$. Although, we will not need it, an explicit form of this rational map is given by: 
\begin{align*} \cp{3} &\to \cp{2} \\
 [u_0:u_1:u_2:u_3] &\to [ u_0 u_2 - u_1^2 : u_0 u_3 - u_1 u_2 : u_1 u_3 - u_2^2]  
\end{align*}
Through each point of $N_{\Delta}$ there is a $\cp{1}$ of secant or tangent lines which are separated by the blow-up. Indeed, blowing-up $N_\Delta$ we obtain the map $p$ which is a $\cp{1}$-bundle over $\cp{2}$.

Under this map the Lagrangian $L_{\Delta}$ is sent to $SU(2)/{Pin}_-(2)=\rp{2}$. Indeed, the restriction of the projection $p: L_{\Delta} \to \rp{2}$ is a circle bundle, where the fibre through $\Delta$ is $\{\exp(\theta\sigma_3)\Delta\ :\ \theta\in[0,2\pi/3]\}$. In particular the points $m'=1$ and $m=\exp(\pi\sigma_3/6)$ are in the same fibre. In fact, we have the following diagram:
\[ \xymatrix{
SU(2)/ C_6 \ar[r] \ar[d] &  SU(2)/ \Gamma_\Delta \ar[d]_{p} \\
SU(2)/S^1 = S^2
\ar[r]
& SU(2)/ Pin_{-} (2) = \rp{2}
} \]
from which one concludes that $L_{\Delta}$ is a circle bundle over $\rp{2}$ with Euler number $\pm 3$.

The divisor $Y_{\Delta}$ is the variety of tangent lines to $N_{\Delta}$ so its
proper transform $\tilde{Y}_{\Delta} \simeq \cp{1} \times \cp{1} $ projects to the locus of double points in
$\OP{Sym}^2\left(N_{\Delta}\right)$, that is the discriminant conic
$\delta\subset\cp{2}$. The exceptional divisor $E_{\Delta} \subset
\widetilde{\cp{3}}$ is such that the restriction \[ p|_{E_{\Delta}} : E_\Delta
\to \cp{2} \] is a double cover branched over $\delta \subset \cp{2}$, hence
$E_\Delta = \cp{1} \times \cp{1}$. One can easily check that (see
\cite{hitchin}) $\tilde{Y}_\Delta + E_\Delta$ is an anticanonical divisor of
$\widetilde{\cp{3}}$. In fact, $\widetilde{\cp{3}}$ is still Fano (it is number 27 in the Mori-Mukai list \cite{MM} of Fano 3-folds with $b_2=2$) and it follows
as before from the long exact sequence of the pair $(\widetilde{\cp{3}},L_\Delta)$, 
\[ 0 \to H_2(\widetilde{\cp{3}};\ZZ)= \ZZ^2 \to H_2(\widetilde{\cp{3}}, L_\Delta; \ZZ) \to H_1(L_\Delta;\ZZ) = \ZZ_4 \to 0 \] 
that $L_{\Delta} \subset \widetilde{\cp{3}}$ is a monotone Lagrangian.

\subsubsection{Lifting and projecting discs}

Now, given a holomorphic disc $u: (D^2, \partial D^2) \to (\cp{3}, L_\Delta)$,
we write \[ \tilde{u} : (D^2 , \partial D^2) \to (\widetilde{\cp{3}}, L_\Delta)
\]  for the holomorphic disc that is obtained by taking the proper
transform of $u$ to $\widetilde{\cp{3}}$. We write \[ p(\tilde{u}) : (D^2,
\partial D^2) \to (\cp{2}, \rp{2}) \] for the projection of $\tilde{u}$ via the
map $p$. Of course, if $u$ misses the twisted cubic $N_\Delta$, one can
directly project via the map $\cp{3} \setminus N_{\Delta} \to \cp{2}$. 

The Maslov index of $u$, $\tilde{u}$ and $p(\tilde{u})$ can be understood via the following formulae :
\[ \mu (u) = 2 [u] \cdot Y_\Delta, \ \ \ \mu(\tilde{u}) = 2 [\tilde{u}] \cdot (\tilde{Y}_\Delta + E_\Delta), \ \ \ \mu(p(\tilde{u})) = 3 [p(\tilde{u})] \cdot \delta \]
These hold because $- K_{\cp{3}} = Y_\Delta$ , $- K_{\widetilde{\cp{3}}} = \tilde{Y}_\Delta + E_\Delta$ and $-K_{\cp{2}} = (3/2) \delta$. 

Furthermore, if $\pi : \widetilde{\cp{3}} \to \cp{3}$ is the blow-down map, we have 
\[ -K_{\widetilde{\cp{3}}} = -\pi^* (K_{\cp{3}}) - E_\Delta \]
Hence, it follows that $\pi^*(Y_\Delta) = \tilde{Y}_\Delta + 2E_\Delta$ which implies:
\begin{equation} \mu(u) = 2 [u] \cdot Y_\Delta = 2 [\tilde{u}] \cdot (\tilde{Y}_\Delta + 2E_\Delta) = \mu(\tilde{u})+ 2 [\tilde{u}] \cdot E_\Delta \end{equation}

Finally, note that $p: \widetilde{\cp{3}} \to \cp{2}$ is flat, therefore
$p^*(\delta) = \tilde{Y}_\Delta$, hence we have the formula: \[ [\tilde{u}] \cdot
\tilde{Y}_\Delta = [p(\tilde{u})] \cdot \delta \]
This implies:
\begin{equation} \label{maslovproj} \mu(\tilde{u}) = 2 [\tilde{u}] \cdot E_\Delta + \frac{2}{3} \mu(p(\tilde{u}))  \end{equation}

\begin{exm}\label{exm:m2}
Maslov 2 discs $u$ in $(\cp{3}, L_\Delta)$ intersect $Y_\Delta \setminus
N_\Delta$ at a unique point transversely, hence their projections $p(\tilde{u})$
intersect the conic $\delta$ transversely at a unique point. As $\rp{2}$ is the
fixed point locus of an anti-holomorphic involution in $\cp{2}$, such discs can
be doubled to rational curves, $D(p(\tilde{u}))$, which intersect the conic
$\delta$ at 2 points, hence are necessarily (real) lines in $\cp{2}$.
Conversely, it is easy to see from our classification of Maslov 2 discs $u:
(D^2, \partial D^2) \to (\cp{3}, L_\Delta)$ from Lemma \ref{class2} that either
half of any real line in $\cp{2}$ is the projection of a unique Maslov 2 disc.  
\end{exm}

\subsubsection{Projection of a non-axial Maslov 4 disc}

To show that there is no non-axial Maslov 4 disc passing through $m'$ and $m$ we will assume there is such a disc and derive a contradiction.

We have argued above that any non-axial Maslov 4 disc $u$ in $\cp{3}$ with
boundary on $L_\Delta$ intersects $Y_{\Delta}$ in a subset of $Y_\Delta \setminus
N_\Delta$; it intersects in either two points transversely or one point
tangentially. We can project such a disc to $\cp{2}$ and the projected disc
$p(\tilde{u})$ therefore intersects $\delta$ in either two points transversely or
one point tangentially.


By assumption, the boundary of our Maslov 4 disc passes through $m'$ and
$m$. Under the projection $p$, these are mapped to the same point
in $\cp{2}$ since they are contained in the fibre of the circle fibration
$p:L_{\Delta} \to \rp{2}$. Thus the doubled curve $D(p(\tilde{u}))$ either has a
real double point or is a double cover. However, $D(p(\tilde{u}))$ is irreducible,
and an irreducible conic cannot have a double point. Therefore, the only
remaining possibility is that $D(p(\tilde{u}))$ is a double cover. Note that
$D(p(\tilde{u}))$ intersects the conic $\delta$ at 4 points (counted with
multiplicity). Hence, it has to be a double cover of a real line $l \simeq
\cp{1}$. Now, \[ D(p(\tilde{u})) : (\cp{1}, \rp{1}) \to (l, l \cap \rp{2}) \subset
(\cp{2}, \rp{2}) \] is a double covering map which is equivariant with respect
to the antiholomorphic involutions. From the Riemann-Hurwitz formula, it is easy to
compute ($2 = 4 - (2-1) + (2-1)$), that there must be exactly 2 branch points and these will have
multiplicity 2. There are two distinct ways this can happen: 

\begin{enumerate}[(a)]
\item The branch points are antipodal and lie in $\cp{1} \setminus \rp{1}$ 
\item The branch points can be any two distinct points in $\rp{1}$
\end{enumerate}	

In case (a), we will show that $u$ is a double cover of an (axial) Maslov 2
disc in $(\widetilde{\cp{3}}, L_\Delta)$, hence it cannot have boundary passing
through $m'$ and $m$ - a fact that follows from our classification
of Maslov 2 discs as we know that the boundaries of Maslov 2 discs are given
by sections of the circle bundle $p: L_\Delta \to \rp{2}$ over real lines in
$\rp{2}$. Finally, we will argue that case (b) cannot occur for any Maslov
$4$ disc $u$ which misses $N_\Delta$. 

\emph{Case (a):} In this case, there is a real line $l$, one half of which is a disc double-covered by $p(\tilde{u})$. As in Example \ref{exm:m2} there is a unique axial Maslov 2 disc $v$ on $(\cp{3},L_{\Delta})$ whose proper transform $\tilde{v}$ projects to this disc in $l$. Therefore we write $p(\tilde{v})\subset l$ for the disc on $(\cp{2},\RR\PP^2)$ which is double-covered by $p(\tilde{u})$. The disc $p(\tilde{v})$ intersects
the discriminant conic $\delta \subset \cp{2}$ at a unique point and the double
$D(p(\tilde{v}))$ is the real line $l$.

Consider the total space of the $\cp{1}$ fibration restricted to the preimage
of the disc $p(\tilde{v})$. Call this $F =
p^{-1}(p(\tilde{v}))$. The intersection $L_\Delta \cap F$ is a Lagrangian in $F$ that is a circle
bundle in $p^{-1} (p(\partial {\tilde{v}}))$. It is easy to see that this is a
Lagrangian Klein bottle $K$ in $F$ as the monodromy is a reflection on the
circle fibre. Thus, $F$ is a $\cp{1}$ fibration over $D^2$ (hence
holomorphically it is $D^2 \times \cp{1}$) and $K$ is a Lagrangian Klein bottle
in $E$ which fibres over $\partial D^2$. 

Now observe that since $\tilde{v}$ is embedded, we have a short exact sequence of Riemann-Hilbert pairs (suppressing the totally real subbundle from the notation):
\begin{equation}\label{eq:rhseq1} 0 \to T \tilde{v}(D^2) \to  T \widetilde{\cp{3}}|_{\tilde{v}(D^2)} \to \nu_{\widetilde{\cp{3}}} (\tilde{v}(D^2)) \to 0\end{equation}
where $\nu_{\cp{3}}$ is the Riemann-Hilbert pair obtained by taking the normal bundle to $\tilde{v}(D^2)$ in $\widetilde{\cp{3}}$ and $\tilde{v}(\partial D^2)$ in $L_\Delta$. This exact sequence implies
\[2=\mu(T\widetilde{\cp{3}}|_{\tilde{v}(D^2)})=\mu(T\tilde{v}(D^2))+\mu(\nu_{\widetilde{\cp{3}}}(\tilde{v}(D^2))=2+\mu(\nu_{\widetilde{\cp{3}}}(\tilde{v}(D^2))\]
so $\mu(\nu_{\widetilde{\cp{3}}}(\tilde{v}(D^2))=0$.

Furthermore, since $p(\tilde{v})$ is embedded, we have a short exact sequence of Riemann-Hilbert pairs:
\begin{equation}\label{exact1} 0 \to \nu_{F} (\tilde{v}(D^2)) \to \nu_{\widetilde{\cp{3}}}(\tilde{v}(D^2)) \to \nu_{\cp{2}} (p(\tilde{v})(D^2)) \to 0 \end{equation}
where $\nu_{F}$ and $\nu_{\cp{2}}$ denote the normal (Riemann-Hilbert) bundles
in $F$ and $\cp{2}$ respectively. (The real subbundle of the complex normal
bundle is given as the normal bundle to $\partial{D}^2$ in $L_{\Delta} \cap F$
and $\rp{2}$ respectively.) Now, $p(\tilde{v})$ is a real line, hence as a disc
in $\cp{2}$ it has Maslov index 3 and, since it is embedded, we have a short
exact sequence of Riemann-Hilbert pairs
\begin{equation}\label{exact2} 0 \to T(p(\tilde{v})(D^2)) \to T\cp{2}|_{p(\tilde{v})(D^2)} \to \nu_{\cp{2}} (p(\tilde{v})(D^2)) \to 0 \end{equation}
which implies
\[3=\mu(T\cp{2}|_{p(\tilde{v})(D^2)})=\mu(T(p(\tilde{v})(D^2)))+\mu(\nu_{\cp{2}} (p(\tilde{v})(D^2)))=2+ \mu(\nu_{\cp{2}} (p(\tilde{v})(D^2)))\]
so $\mu(\nu_{\cp{2}} (p(\tilde{v})(D^2))=1$. Finally, Equation \eqref{exact1} gives
\[0=\mu(\nu_{\widetilde{\cp{3}}}(\tilde{v}(D^2)))=\mu(\nu_{F} (\tilde{v}(D^2)))+\mu(\nu_{\cp{2}} (p(\tilde{v})(D^2)))\]
Therefore, $\mu(\nu_F(\tilde{v}(D^2)))=-1$. Considered as a disc inside $(F,K)$ the Riemann-Hilbert pair of $\tilde{v}(D^2)$ fits into an exact sequence
\[0\to T(\tilde{v}(D^2))\to TF|_{\tilde{v}(D^2)}\to \nu_F(\tilde{v}(D^2))\to 0\]
hence the Maslov index of the holomorphic
disc $\tilde{v}$ viewed in $(F, K)$ is
\[\mu(TF|_{\tilde{v}(D^2)})=\mu(T(\tilde{v}(D^2)))+\mu(\nu_F(\tilde{v}(D^2)))=2-1=1.\]

On the other hand, we can double the projective bundle $(F,K) \to (D^2,
\partial D^2)$ to a $\cp{1}$-bundle, $D(F) \to \cp{1}$, with an antiholomorphic
involution $\iota : D(F) \to D(F)$ such that $Fix(\iota) = K$.
\footnote{Dangerous bend: $D(F)$ no longer embeds in $\widetilde{\cp{3}}$.} The
construction of $(D(F), \iota)$ can be described as follows: recall that $F$ is
a holomorphically trivial $\cp{1}$-bundle over $D^2$ and $K$ fibres
over $\partial D^2$ with fibres given by equatorial circles in a fibre of $F$.
Therefore, $K$ defines a fibrewise antiholomorphic involution $\iota_K$ on the
restriction of $F$ over $\partial D^2$. To construct $D(F)$ one takes another
copy of $(F,K)$ with the complex conjugate holomorphic structure, which we write as  $(\overline{F},K)$, and glues them above $\partial D^2$ using $\iota_K$ which then gives us a holomorphic $\cp{1}$-bundle over $\cp{1}$:  \[ D(F)
= (F, K) \cup_{\iota_K} (\overline{F}, K) \] It is now clear that the
involution $\iota_K$ extends to $D(F)$ to give an involution $\iota$ with $Fix(\iota)= K$ which acts on the base $\cp{1}$ by the usual complex conjugation.

The holomorphic disc $\tilde{v}(D^2)$ doubles to give a section $D(\tilde{v})$ of
$D(F)$ whose self-intersection number is equal to $\mu(\nu_F(\tilde{v}(D^2)))=-1$. Thus $D(F)$ is the Hirzebruch surface $\PP(\mathcal{O} \oplus \mathcal{O}(1))$. The homology of $D(F)$ is therefore spanned by two classes $s,f$ with $s^2=-1$, $f^2=0$, $f\cdot s=0$ and $[D(\tilde{v})]=s$.

Crucially, $E_\Delta \cap F$ also doubles since $E_\Delta$ intersects the fibres over $\partial{D^2}$ at antipodal points, which are exchanged by $\iota_K$. Thus we obtain a divisor $D(E_\Delta \cap F)$ in $D(F)$ which intersects a generic fibre in two points. It is also disjoint from $D(\tilde{v})$ because Maslov 2 discs in $(\cp{3},L_{\Delta})$ are disjoint from the twisted cubic $N_{\Delta}$ and because $E_{\Delta}$ is the exceptional divisor for blow-up along $N_{\Delta}$. These intersection numbers imply that $[D(E_{\Delta}\cap F)]=2f+2s$.

Finally, the holomorphic disc $\tilde{u}$ doubles to give a holomorphic curve $D(\tilde{u})$. This curve is disjoint from $D(E_{\Delta}\cap F)$ because the curve $u$ is disjoint from the twisted cubic. Therefore, if $[D(\tilde{u})]=as+bf$, we have
\[(as+bf)\cdot(2f+2s)=0=2b\]
so $[D(\tilde{u})]=as$. Note that $a>0$ as $D(\tilde{u})$ intersects the fibre $f$ positively. The curves $D(\tilde{u})$ and $D(\tilde{v})$ have negative intersection
\[D(\tilde{u})\cdot D(\tilde{v})=-a\]
so their images must coincide by positivity of intersections. In particular, the image of $\tilde{u}$ coincides with the image of $\tilde{v}$ and so $\tilde{u}$ is a double cover of $\tilde{v}$ as required.

\emph{Case (b):} In this case, let $Q$ be the preimage $p^{-1}(l)$ and $K=L_{\Delta}\cap Q$. Formally, the argument is similar to case (a) except that instead of viewing the disc $\tilde{u}$ as a map to $(D(F),K)$, we will see it as a map to $(Q,K)$. We recall that the $\cp{1}$-bundle, $p:
\widetilde{\cp{3}} \to \cp{2}$ arises from a construction of Schwarzenberger
\cite{schwarzenberger} of rank 2 vector bundles on $\cp{2}$. Namely, let \[ q :
\cp{1} \times \cp{1} \to \cp{2} \] be the double branched covering over the conic
$\delta \in \cp{2}$. Schwarzenberger considers the rank 2 bundle \[ \mathcal{E}
= q_* (\mathcal{O}(3,0))  \] where $\mathcal{O}(3,0)$ denotes the unique
holomorphic line bundle of bidegree $(3,0)$ on $\cp{1} \times \cp{1}$. As is
explained in \cite{hitchin} (see also \cite{Nak}), the projectivisation
$\PP(\mathcal{E})$ of this bundle is the $\cp{1}$-bundle $\widetilde{\cp{3}}
\to \cp{2}$. It is also proved in \cite[Proposition 8]{schwarzenberger} that if one restricts $\mathcal{E}$ to a line $l \subset \cp{2}$, then:	
$$\mathcal{E}|_{l} \simeq \begin{cases} \mathcal{O}(2) \oplus \mathcal{O} & \text{if \  } l \text{\ is 
	tangent to  }\delta \\ \mathcal{O}(1)\oplus \mathcal{O}(1) & \text{otherwise.}
\end{cases}$$
Since we have defined $Q$ as the preimage of a real line $l$ (which cannot be tangent to $\delta$), it follows that $Q$ is isomorphic to the projectivisation $\PP(\mathcal{O}(1) \oplus \mathcal{O}(1)) \simeq \cp{1} \times \cp{1}$. 
Therefore we can view $\tilde{u}$ as a holomorphic map \[ \tilde{u} : (D^2 ,
\partial D^2)  \to ( \cp{1} \times \cp{1} , K) \subset (\widetilde{\cp{3}} ,
L_\Delta).\]  $K$ is again a Lagrangian Klein bottle, as the monodromy of
$L_\Delta \cap Q \to S^1$ is a reflection on the circle fibre. 

First note that since $p(\tilde{u})$ is
not immersed (at the two boundary points), we cannot immediately apply the
Maslov index computation from case (a) as we do not have the exact sequences (\ref{exact1}), (\ref{exact2}). On the other hand, $p(\tilde{u})$ is a
smooth limit of embedded real conics $c_t$ in $\cp{2}$ (see Figure \ref{fig:case-b}). (Explicitly, in suitable coordinates, it can be
exhibited as a limit of a family of the form $c_t = \{ (x:y:z) \in \cp{2} :  x^2/a^2 + y^2/t^2 =z^2 \} $ as $t \to 0$). 
Therefore, \[ \mu(\nu_{\widetilde{\cp{3}}/Q}(\tilde{u})) = \mu(\nu_{\cp{2}}(c_t)) = \mu(T\cp{2}|_{c_t}) - \mu(Tc_t) = 6- 2=4 \]
\begin{figure}[htb]
\begin{center}
\includegraphics{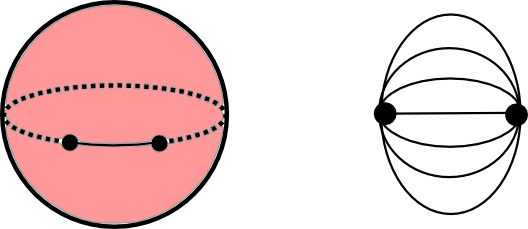}
\caption{Case (b). Left: The sphere is $\cp{1}$ and the dashed equator is $\RR\PP^1\subset\cp{1}$. The shaded area is the image of $p(\tilde{u})$, a holomorphic disc covering the whole of $\cp{1}$; its boundary circle maps two-to-one onto an interval in $\RR\PP^1$ (the non-dashed interval in the figure) with two critical points. Right: This is a picture of the boundaries of discs inside $\RR\PP^2$; $p(\tilde{u})$ arises as a limit of real ellipses $c_t$ whose boundaries are embedded.}
\label{fig:case-b}
\end{center}
\end{figure}
Now as in case (a),  we can compute $\mu(\nu_Q (\tilde{u}(D^2)) = -2$. Hence, the Maslov index of the (embedded) holomorphic disc $\tilde{u}$ viewed in $(Q, K)$ is 0. We write this as: \[ \mu_Q (\tilde{u}) = 0. \]
The long exact sequence of
the pair $(Q,K)$ gives
\[ 0 \to H_2(Q; \ZZ)= \ZZ^2 \to H_2(Q, K; \ZZ) \to H_1(K;\ZZ) =
\ZZ \oplus \ZZ_2 \to 0 .\] 

Now, let $\tilde{v}$ be an (axial) Maslov 2 holomorphic disc such that
$p(\tilde{v})$ is one half of the base real line $l$. (As we have
mentioned several times, the existence of this follows from our classification
of Maslov 2 discs.) Let $\tilde{w}$ be a Maslov 2 disc that lies on a fibre of
the projection $p|_{Q} : Q \to \cp{1}$ . Such discs $\tilde{w}$ are obtained as
the proper transform of axial Maslov 4 discs $w$ in $\cp{3}$. We now observe
that $\partial{\tilde{v}}$ and $\partial{\tilde{w}}$ project to generators of
$H_1(K)$ - they give generators for the summands $\ZZ$ and $\ZZ_2$ respectively.
We can compute, as in case (a), that the Maslov index $\mu_Q(\tilde{v}) =1$ when $\tilde{v}$ is viewed as a holomorphic disc mapping to $(Q,K)$. 
Let $f \simeq \cp{1}$ be a fibre of the projection $p|_{Q} : Q \to \cp{1}$ and
$s \simeq \cp{1}$ be a section such that the anticanonical divisor of $Q$ is
given by \[ -K_Q = 2f +2s \] 

We infer from the above exact sequence that the
elements $\tilde{v},\tilde{w},f$ and $s$ generate $H_2(Q, K ; \ZZ)$. From the description of $-K_Q$, we can compute that we have $\mu_Q([f]) = \mu_Q([s])= 4$ and $\mu_Q(\tilde{w}) = 2$. The latter follows because $\tilde{w}$ is a disc that lies in the fibre of the projection $p|_{Q}: Q \to \cp{1}$ and since $K$ intersects this fibre at the equator, the disc $\tilde{w}$ can be reflected in the fibre. Thus, we have $2 \mu_Q(\tilde{w}) = \mu_Q ([f])=4$. 

Now, we can write 
\[ [\tilde{u}] = a [f] + b [s] + c [\tilde{v}] + d [\tilde{w}] \]
for some integers $a,b,c,d \in \ZZ$. We will now use the fact that $\mu_Q(\tilde{u})=0$ and that $\tilde{u}$ does not intersect the divisor $E_\Delta \cap Q$ to arrive at a contradiction. Indeed we have
\[ 0= \mu_Q (\tilde{u}) = 4a + 4b + c +2d . \]
To compute $[\tilde{u}] \cdot [(E_\Delta \cap Q)]$, observe first that from the geometric situation, we deduce immediately that:
\[ [f] \cdot (E_\Delta \cap Q) = 2 , \ \  [\tilde{v}] \cdot (E_\Delta \cap Q) = 0 , \ \ [\tilde{w}] \cdot (E_\Delta \cap Q) =1 \]
It remains to compute $[s] \cdot (E_\Delta \cap Q)$. To this end, we know that $\mu_Q([s])= 4$ and $s$ and $p(s) \simeq \cp{1}$ are embedded hence we can use the short exact sequence of Riemann-Hilbert bundles
\[0\to TQ|_s\to T\widetilde{\cp{3}}|_s\to \nu_{\cp{2}}(p(s))\to 0\]
to compute that $\mu([s])= 6$ when $s$ is viewed as a holomorphic map to $\widetilde{\cp{3}}$. Next, we use the Formula (\ref{maslovproj}) to compute:
\[ 6 = \mu([s]) = 2 [s] \cdot (E_\Delta\cap F) + 4 \]
so $[s]\cdot(E_{\Delta}\cap F)=1$. Therefore, we have:
\[ 0= [\tilde{u}] \cdot [(E_\Delta \cap Q)] = 2a + b +d \]
Hence, putting $\mu_Q(\tilde{u})= 0 $ and $[\tilde{u}] \cdot [(E_\Delta \cap Q)] =0$ give us the equality 
\[ 2b +c = 0\] 
On the other hand, since $p(\tilde{u})$ is injective along its interior, in particular , it intersects the fibres above any point in $ l \setminus (l \cap \rp{2})$ at a unique point. Let $f_0$ and $f_{\infty}$ represent two such fibres corresponding to the two components of $l \setminus (l \cap \rp{2})$. Calculating  $[\tilde{u}] \cdot [f_0]=[\tilde{u}] \cdot [f_\infty] = 1$ gives the constraints:
\[ b+c =1 \text{\ \ and \ \ }  b =1 \]
which contradicts $2b+c=0$. Hence, we conclude that $\tilde{u}$ (and thus $u$) could not have existed.

\section{Floer cohomology}

\subsection{Eigenvalues of the first Chern class}

\begin{lma}\label{lma:modp}
Let $\kk$ be a field with characteristic $p\neq 2$. If $HF(L_{\Delta},L_{\Delta};\kk)\neq 0$ then $p=5$ or $p=7$.
\end{lma}
\begin{proof}
By Proposition \ref{prp:c1-eval}, the Floer cohomology over a field is nonzero only if $\mathfrak{m}_0(L_{\Delta})$ is an eigenvalue of $c_1(\cp{3})$ acting by quantum product on $QH^*(\cp{3})$. Since the characteristic polynomial for the quantum action of $c_1(\cp{3})$ given in Figure \ref{fig:c1-char-poly} is $\lambda^4-256$ and since $\mathfrak{m}_0(L_{\Delta})=3$ we need to work in $\ZZ/p$ such that $3^4-256=-5^2\times 7\equiv 0\mod p$. Thus $p=5$ or $p=7$.
\end{proof}

\subsection{Computing the Floer differential}

Let us choose an orientation and a spin structure on $L_{\Delta}$. We use this choice to orient the moduli spaces of holomorphic discs.

We use the Biran-Cornea pearl complex \cite{BC0} to compute the Floer cohomology $HF(L_{\Delta},L_{\Delta};\ZZ)$. The cochain groups are generated by the critical points of our Morse function over $\ZZ$ and are $\ZZ/2$-graded by the parity of the Morse indices. The Floer differential of a Morse cochain $c$ is
\[d_{F}c=d_{M}c+\sum_b\pm P(c,b)b\]
where $d_M$ is the Morse differential, the sum is over critical points $b$ and the coefficient $P(c,b)$ counts pearly trajectories connecting $c$ to $b$. A pearly trajectory is a combination of upward Morse flow lines and holomorphic discs. The sign conventions (and orientations on the pearly moduli spaces) are worked out in \cite[Appendix A]{BC1}.

This definition presupposes a choice of Morse function, metric and almost complex structure. We will use the standard complex structure on $\cp{3}$ and the round metric on $L_{\Delta}$. We will need to perturb the Morse function slightly from the one we constructed earlier to ensure transversality between the holomorphic discs and the Morse flow lines: currently the boundaries of the Maslov 2 holomorphic discs through the maximum $m$ run along the gradient flow lines $\alpha_i$ and those through the minimum $m'$ run along the flow lines $\alpha_i'$. Conveniently, all the Maslov 2 holomorphic discs through $m'$ stay in the lower handlebody $N'$ of the Heegaard decomposition and all the Maslov 2 holomorphic discs through $m$ stay in $N$.

\begin{lma}
We can make a small perturbation of the Morse function such that:
\begin{itemize}
\item the descending discs $D_i$ of the index two critical points $x_i$ are unchanged,
\item the ascending discs $D_i'$ of the index one critical points $x_i'$ are unchanged,
\item the ascending lines from $x_i$ are disjoint from the Maslov 2 discs through $m$,
\item the ascending lines from $x_i'$ are disjoint from the Maslov 2 discs through $m'$.
\end{itemize}
Moreover we can ensure that the perturbed Morse function is still invariant under the left $\ZZ/3$-action on $SU(2)\Gamma_{\Delta}$.
\end{lma}

\begin{center}\includegraphics{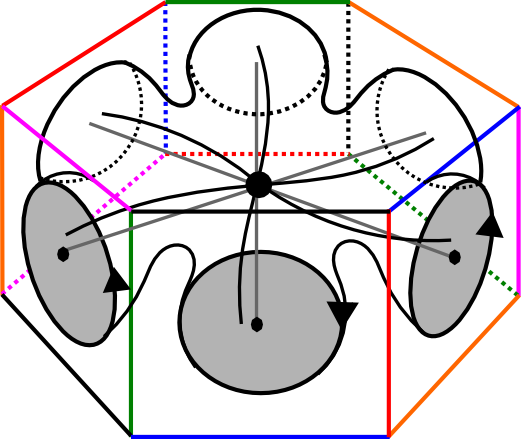}\end{center}

\begin{proof}
The perturbation which needs to be made is local near the loops $\alpha_i$ and $\alpha'_i$. It is effected by pulling back along a diffeomorphism $\phi$ which is supported in a neighbourhood of these loops. Denote by $\phi_i$ (respectively $\phi'_i$) the restriction of $\phi$ to a neighbourhood of $\alpha_i$ (respectively $\alpha'_i$). Consider a diffeomorphism $\phi_i$ (respectively $\phi'_i$) which bends the loop $\alpha_i$ (respectively $\alpha'_i$) so that it intersects the original loop only at the critical point $m$ (respectively $m'$) but still intersects $D_i'$ (respectively $D_i$) once transversely. The $\ZZ/3$-action rotates $\alpha_i$ to $\alpha_{i+1}$ (respectively $\alpha'_i$ to $\alpha'_{i+1}$) and we can simply choose $\phi_2$, $\phi_3$ to be conjugates of $\phi_1$ under this action (similarly for $\phi'_i$), so the resulting diffeomorphism $\phi$ is $\ZZ/3$-invariant. Pulling the Morse function back along $\phi$ yields a $\ZZ/3$-symmetric Morse function but now the gradient flow lines do not run along the boundaries of discs.
\end{proof}

We can also ensure that the following choices are invariant under the $\ZZ/3$-action:
\begin{itemize}
\item the orientations of the ascending and descending manifolds of $x_i$ and $x_i'$;
\item the orientations of the boundaries $\alpha_i$ and $\alpha_i'$ of the Maslov 2 discs through $m$ and $m'$.
\end{itemize}
This is important because it means whatever choices of orientations we make, the Floer complex will be cyclically symmetric under permuting $x_i$ and the $x_i'$.

The following lemma will be useful in establishing transversality. Let $V$ be the standard representation of $SU(2)$, identify $\cp{3}$ with $\PP(\OP{Sym}^3V)$ and pick coordinates $(x,y)\in V^*$ so that we can consider cubic polynomials in $x,y$ as defining points in $\cp{3}$.

\begin{lma}\label{lma:tang-rep}
Consider the points $x^3$ and $y^3$ in the twisted cubic $N_{\Delta}$ and the subgroup $S^1\subset SU(2)$ consisting of rotations of $\cp{1}$ which fix $0,\infty\in\cp{1}$ and hence fix $x^3,y^3\in N_{\Delta}$. The tangent space $T_{x^3}\cp{3}$ splits into weight spaces for the $S^1$-action with weights $2,4,6$. In particular the action of $S^1$ has no fixed vector in $T_{x^3}\cp{3}$.
\end{lma}
\begin{proof}
The complex lines connecting $x^3$ to $x^2y$, $x^2y$, $y^3$ span $T_{x^3}\cp{3}$. They are invariant under the $S^1$-action and come with weights $2,4,6$ respectively. Note that under the corresponding subgroup of rotations in $SO(3)$ the weights are $1,2,3$ but weights are doubled for the spin-preimage of $S^1$.
\end{proof}

Now we can compute the Floer complex.
\begin{lma}\label{lma:floer-cpx}
The Floer differential is given by
\begin{align*}
d_Fm&=Y(x_1+x_2+x_3)+2Zm'\\
d_Fm'&=0\\
d_Fx_1'&=x_1+x_2+2x_3+Xm'\\
d_Fx_2'&=2x_1+x_2+x_3+Xm'\\
d_Fx_3'&=x_1+2x_2+x_3+Xm'.
\end{align*}
for some $X,Y,Z\in\{-1,1\}$.
\end{lma}
\begin{proof}
The coefficient of $x_i$ in $d_Fm$ is the (signed) count of pearly trajectories consisting of a Maslov 2 disc through $m$ which intersects the descending manifold of $x_i$. The only such disc has boundary $\alpha_i$ which intersects the descending manifold once transversely so the coefficient of $x_i$ in $d_Fm$ is $Y_i\in\{-1,1\}$. By cyclic symmetry, $Y_1=Y_2=Y_3=\colon Y$.

The coefficient of $m'$ in $d_Fm$ is the (signed) count of pearly trajectories consisting of a Maslov 4 disc whose boundary contains both $m'$ and $m$. There are two such discs and they contribute with the same sign (see Remark \ref{rmk:signs} below) so the coefficient is $2Z$. Note that by Lemma \ref{lma:ev2-transversality} this pearly moduli space is regular: it suffices to check that the $S^1$-action which rotates the Maslov 4 disc around its centre $u(0)$ has no fixed vector in its action on $T_{u(0)}\cp{3}$. This follows from Lemma \ref{lma:tang-rep}.

The coefficient of $m'$ in $d_Fx_i'$ is the (signed) count of pearly trajectories consisting of an upward flowline from $x_i'$ which intersects a Maslov 2 disc through $m'$. There is precisely one of these, given by the intersection of the boundary $\alpha'_i$ with the ascending disc $D_i'$, so the coefficient is $X_i\in\{-1,1\}$. By cyclic symmetry, $X_1=X_2=X_3=\colon X$.
\end{proof}

\begin{cor}
We also have $d_Fx_i=0$, $i=1,2,3$.
\end{cor}
\begin{proof}
Certainly the Morse differentials of $x_i$ vanish; suppose that $d_Fx_i=px_1'+qx_2'+rx_3'$ then we get
\[d_F^2x_1=p(x_1+x_2+2x_3)-q(2x_1+x_2+x_3)-r(x_1+2x_2+x_3)+Xm'(p+q+r)=0\]
so $p-2q-r=p-q-2r=2p-q-r=p+q+r=0$. These equations imply $p=q=r=0$. By cyclic symmetry $d_Fx_1=d_Fx_2=d_Fx_3=0$.
\end{proof}

\begin{cor}
The Floer differential $d_F\colon CF^0(L_{\Delta},L_{\Delta};\ZZ)\to CF^1(L_{\Delta},L_{\Delta};\ZZ)$ vanishes. The matrix of the Floer differential $d_F\colon CF^1(L_{\Delta},L_{\Delta};\ZZ)\to CF^0(L_{\Delta},L_{\Delta};\ZZ)$ with respect to the bases $m,x_1',x_2',x_3'$ and $m',x_1,x_2,x_3$ is
\[\left(\begin{array}{cccc}
Y & Y & Y & 2Z\\
1 & 1 & 2 & X\\
2 & 1 & 1 & X\\
1 & 2 & 1 & X
\end{array}\right).\]
\end{cor}

\begin{cor}
We have:
\begin{enumerate}
\item $HF^*(L_{\Delta},L_{\Delta};\ZZ)\neq 0$; in fact $HF^0(L_{\Delta},L_{\Delta};\ZZ)\cong\ZZ/5$, $HF^1(L_{\Delta},L_{\Delta};\ZZ)=0$.
\item Moreover if $\kk$ is a field of characteristic 5 then
\[HF^0(L_{\Delta},L_{\Delta};\kk)\cong HF^1(L_{\Delta},L_{\Delta};\kk)\cong\kk.\]
\end{enumerate}
\end{cor}
\begin{proof}
The determinant of $d_F\colon CF^1(L_{\Delta},L_{\Delta};\ZZ)\to CF^0(L_{\Delta},L_{\Delta};\ZZ)$ is $8Z-3XY$. Since $X,Y,Z\in\{-1,1\}$ this is not a unit in $\ZZ$ and hence the matrix has trivial kernel (so $HF^1(L_{\Delta},L_{\Delta};\ZZ)=0$) but nontrivial cokernel of size $|8Z-3XY|$.

Indeed if we work over a field $\kk$ of characteristic $p$ where $p$ divides $|8Z-3XY|$ then the Floer cohomology is $\kk$ in degrees zero and one. By Lemma \ref{lma:modp}, $|8Z-3XY|$ must be zero modulo 5 or 7. The only possibility is $p=5$ (and $XY/Z=1$).
\end{proof}

\begin{rmk}\label{rmk:signs}
Note that we could also argue this way to show that the two Maslov 4 discs contribute with the same sign to the Floer differential: otherwise the determinant of $d_F$ would be $\pm 3$ and the Floer cohomology would be nonzero over $\ZZ/3$.
\end{rmk}

\section{Split-generating the Fukaya category}

In this section we show that the Chiang Lagrangian $L_\Delta$, when equipped with various $(\ZZ/5)^\times$-local systems, split-generates the Fukaya category of $\cp{3}$
over $\kk  = \ZZ/5$ (this holds more generally over any field of characteristic 5). We
will use this information to determine the ring structure on $HF^*(L_\Delta,
L_\Delta;\kk)$ indirectly. Furthermore, we will prove that the $A_\infty$ structure
on $HF^*(L_\Delta, L_\Delta;\kk)$ is formal. 

\subsection{The Clifford torus}

Recall from above that we have a (topological) decomposition of $\cp{3}$ as : \[ \cp{3} = T^*
L_\Delta \cup Y_{\Delta}.\] Let us recall a, perhaps more familiar, decomposition
of $\cp{3}$ coming from its toric structure. Namely, we have the action of the
algebraic torus $G= (\CC^*)^3$ on $\cp{3}$ given by:
\[ (t_1,t_2,t_3) \cdot [u_0 : u_1 : u_2 : u_3] = [u_0 : u_1 t_1 : u_2 t_2 : u_3
t_3]. \] 
The action of the compact group $K= (S^1)^3$ is Hamiltonian with moment
map: \[ \mu(u_0: u_1: u_2: u_3) = \frac{1}{2} \left( \frac{|u_1|^2}{\sum_{i=0}^3 |u_i|^2},
\frac{|u_2|^2}{\sum_{i=0}^3 |u_i|^2}, \frac{|u_3|^2}{\sum_{i=0}^3 |u_i|^2 )} \right). \]
Since $K$ is abelian, each fibre of $\mu : \cp{3} \to \RR^3$ is an isotropic
torus. There is a fibre given by 
\[ T_{Cl} = \{ [u_0: u_1: u_2: u_3] : |u_0| = |u_1| = |u_2| = |u_3| \} \]
which is special as it is a monotone Lagrangian (with minimal Maslov number 2).
It is called the Clifford torus. We have a
(topological) decomposition: \[ \cp{3} = T^* T_{Cl} \cup D \] where $D$ is the
toric divisor (union of lower dimensional orbits); $D$ is anticanonical. 

Floer cohomology of the Clifford torus was computed additively by Cho in
\cite{cho1}. When $T_{Cl}$ is equipped with the standard spin structure, one
has $\mathfrak{m}_0(T_{Cl})=4$ (there are four families of Maslov 2 discs corresponding 4
faces of the moment polytope) and there is an (additive) isomorphism
$HF^*(T_{Cl}, T_{Cl};\kk) \simeq H^*(T_{Cl})$. On the other hand, it is shown in
\cite{cho2} that the multiplication on $HF^*(T_{Cl}, T_{Cl};\kk)$ is deformed. More
precisely, one has: \[ HF^*(T_{Cl}, T_{Cl};\kk) \simeq Cl(V, q) \] where $V$ is a
3-dimensional vector space and $Cl(V,q)$ is the Clifford algebra associated
with the quadratic form given by the symmetric matrix:
\[q = \left(\begin{array}{ccc} 2 & 1 & 1 \\ 1 & 2 & 1\\ 1 & 1 & 2
\end{array}\right). \]

Recall that $Cl(V,q)$ is the graded $\kk$-algebra given by the quotient of the
tensor algebra $T(V)$ (where $V$ sits in grading 1) by the two-sided ideal
generated by the elements of the form \[ v \otimes w + w \otimes v - q(v,w) 1. \] 
Note that $Cl(V,0) \simeq H^*(T_{Cl})$ is just the exterior algebra.  It is easy to verify from
the arguments given in \cite{cho1}, \cite{cho2} that this computation remains
valid over a field $\kk$ of characteristic 5. We note that the above quadratic
form is non-degenerate also over $\kk$ of characteristic 5. 

Let $X$ be a closed monotone symplectic manifold. We have a natural splitting of $QH^*(X)$ as a ring:
\[ QH^*(X) = \bigoplus_{\lambda \in Spec(c_1\star)} QH^*(X,\lambda) \]
into generalized eigenspaces for the linear transformation $c_1(X) \star : QH^*(X) \to QH^*(X)$. The Fukaya category $\mathcal{F}(X)$ also splits into mutually orthogonal subcategories 
\[ \mathcal{F}(X) = \bigoplus_{\lambda \in Spec(c_1\star)} \mathcal{F}(X;\lambda) \]
where objects of $\mathcal{F}(X; \lambda)$ are closed, orientable, spin monotone Lagrangian submanifolds $L \subset X$ such that $HF^* (L,L;\kk) \neq 0$ and $\mathfrak{m}_0(L) = \lambda$. (Note that any Lagrangian with non-vanishing Floer cohomology has $\mathfrak{m}_0(L) \in Spec(c_1\star)$ by Proposition \ref{prp:c1-eval} ).

We will make use of the following version of Abouzaid's split-generation criterion \cite{abouzaid} (proved in the monotone setting by Ritter and Smith \cite{RS}, Sheridan \cite{S} and in general by \cite{AFOOO}). 

\begin{thm} \label{generation} (\cite{S} Corollary 3.8) Let $L \subset X$ be a monotone Lagrangian submanifold in a monotone symplectic manifold $M$ with $\mathfrak{m}_0(L)= \lambda$. Suppose the closed-open string map
	\[ \mathcal{CO} : QH^{\bullet}(X,\lambda) \to HH^{\bullet} (CF^*(L,L;\kk)) \]
	is injective, then $L$ split-generates $\mathcal{F}(X; \lambda)$.  
\end{thm}

In the above, $HH^{\bullet}(CF^*(L,L;\kk))$ refers to the Hochschild cohomology of the
$A_\infty$ algebra $CF^*(L,L;\kk)$. Note that in the monotone setting
$\mathcal{F}(X)$ is a $\ZZ/2$ graded $A_\infty$ category. Therefore, $HH^\bullet
(CF^*(L,L;\kk))$ should be computed in the $\ZZ/2$ graded sense \cite{kassel}.

By projection to the 0th order term of the Hochschild complex, we have a ring map:
\[ HH^{\bullet} (CF^*(L,L;\kk)) \to HF^{\bullet} (L,L;\kk). \]
The composition 
\[ QH^{\bullet}(X, \lambda) \to HH^{\bullet} (CF^*(L,L;\kk)) \to HF^{\bullet} (L,L;\kk) \] 
sends the projection of $2c_1(X)$ in $QH^{\bullet} (X,\lambda)$ to $2\mathfrak{m}_0(L) 1$ \cite[Lemma 3.2]{S}. Therefore, in the case when $QH^{\bullet}(X,\lambda)$ has rank 1, it follows immediately from 
Theorem \ref{generation} that a Lagrangian $L$ with non-trivial Floer cohomology and $\mathfrak{m}_0(L) \neq 0$, split-generates the corresponding summand of the Fukaya category.

Let us now restrict our attention to $X= \cp{3}$ and work over the field $\kk$ of characteristic 5. Then, we have $Spec(c_1\star) = \{1,2,3,4\}$ and we have the decomposition:
\[ QH^*(\cp{3}) = \bigoplus_{i \in \{1,2,3,4 \} } QH^*(\cp{3}, i) \]
where for all $i \in \{1,2,3,4 \}$, we have $QH^*(\cp{3}, i) \simeq \kk$.

Thus, Theorem \ref{generation} immediately gives that $T_{Cl}$ (equipped with
its standard spin structure so that $\mathfrak{m}_0(T_{Cl}) =4$) split-generates
$\mathcal{F}(\cp{3}, 4)$. In order to access the other components, we may equip
$T_{Cl}$ with a $(\ZZ/5)^\times$ local system. To this end, we recall from \cite{cho1}
the classification of Maslov 2 discs for $T_{Cl}$. There are four families of
Maslov 2 discs with boundary on $T_{Cl}$. If we fix the point $p= [1:1:1:1] \in
T_{Cl}$, the 4 Maslov 2 discs through this point are given by: \[ \{ [z:1:1:1]
, [1:z:1:1] , [1:1:z:1], [1:1:1:z] : |z| \leq 1 \} \] One obtains all other Maslov 2 discs by
translating these using the torus action. In particular, note that the homology
classes $l_1, l_2, l_3, l_4$ of the boundaries of these discs satisfy:
\begin{equation}\label{constraint} l_1 + l_2 + l_3 + l_4 = 0  \in H_1(T_{Cl}) \end{equation}
and we have $H_1(T_{Cl}) = \ZZ l_1 \oplus \ZZ l_2 \oplus \ZZ l_3$. It follows
from Cho's computation that if we equip $T_{Cl}$ with a local system
$\alpha_{\zeta} : H_1(T_{Cl}) \to (\ZZ/5)^\times$ such that $\alpha_{\zeta} (l_i) =
\zeta$ for some fixed $\zeta \in (\ZZ/5)^\times$ and $i=1,2,3,4$ (note that this is allowed in view
of Equation \eqref{constraint} since $\zeta^4=1$), then we get 
\[ HF^*((T_{Cl}, \zeta), (T_{Cl}, \zeta);\kk) = Cl(V, q_{\zeta}) \text{\ for \ } q_{\zeta} = \zeta q \] and
$\mathfrak{m}_0(T_{Cl}, \zeta)= 4\zeta$.

By abuse of notation, we use $\zeta$ to denote the local system $\alpha_\zeta$. In fact, it is easy to see that these are the only local systems that give non-vanishing Floer cohomology. To summarize, Cho's calculations from \cite{cho1}, \cite{cho2} put together with the split-generation Theorem \ref{generation} leads to:

\begin{cor} \label{cliff} $T_{Cl}$ when equipped with the local system $\zeta=1,2,3,4$ split-generates the summand $\mathcal{F}(\cp{3}, 4\zeta )$ over a field $\kk$ of characteristic $5$. 
\end{cor}

\subsection{The Chiang Lagrangian}

It follows from our computations from the previous sections that Chiang
Lagrangian $L_\Delta$ gives yet another split-generator for the Fukaya category
$\mathcal{F}(\cp{3})$. Namely, a $(\ZZ/5)^\times$-local system $\beta_\zeta :
H_1(L_\Delta) \to (\ZZ/5)^\times$ is determined by a choice of monodromy
$\zeta\in(\ZZ/5)^\times$ for the generator in $H_1(L_{\Delta};\ZZ)=\ZZ/4$. Again by
abuse of notation we will use $\zeta$ to denote the local system $\beta_\zeta$.
The resulting Floer differential $d_F$ gets weighted by $\zeta$ to the
contribution from Maslov 2 discs and $\zeta^2$ for the contribution from Maslov
4 discs so the determinant of $d_F$ becomes \[5\zeta^2\equiv 0\mod 5\] hence
the Floer cohomology is still nonzero over a field of characteristic $5$.

The $\mathfrak{m}_0(L_{\Delta})$ term also picks up a factor of $\zeta$ from the local system and hence we get $\mathfrak{m}_0(L_{\Delta},\zeta)=3\zeta\in\ZZ/5$. As $\zeta$ varies over $(\ZZ/5)^\times$, $\mathfrak{m}_0(L_{\Delta},\zeta)$ takes on all the values $1,2,3,4\in(\ZZ/5)^\times$. These are the fourth roots of unity modulo 5 and hence they are all the possible eigenvalues of $c_1(\cp{3})\star$.

\begin{cor} \label{chi} $L_{\Delta}$ when equipped with the local system $\zeta=1,2,3,4$ split-generates the summand $\mathcal{F}(\cp{3}, 3\zeta )$ over a field $\kk$ of characteristic $5$. 
\end{cor}

Corollary \ref{cliff} and \ref{chi} tell us that there is an $A_\infty$ quasi-equivalence between the categories of $A_{\infty}$-modules
\begin{equation} \label{qiso} CF^*((T_{Cl}, 2\zeta), (T_{Cl}, 2\zeta);\kk)^{mod} \ \cong CF^*( (L_{\Delta}, \zeta), (L_{\Delta}, \zeta);\kk)^{mod} \end{equation} 

Now, since $HF^*((T_{Cl}, 2\zeta), (T_{Cl}, 2\zeta);\kk) = Cl(V, q_{2\zeta})$ is a Clifford algebra with non-degenerate quadratic from $q_{2\zeta}$, it follows from the computation given in \cite{kassel} that:
\[ HH^{\bullet} (HF^*((T_{Cl}, 2\zeta), (T_{Cl}, 2\zeta);\kk)) = HH^{0} (HF^*((T_{Cl}, 2\zeta), (T_{Cl}, 2\zeta);\kk)) =  \kk \]
supported in degree $\bullet = 0$. The Hochschild cochain complex for the $A_\infty$ algebra $CF^*(L,L)$ has a filtration by length of the cochains \cite[Section 1f]{seidelbook} which leads to a spectral sequence
\[ HH^{\bullet} (HF^*(L,L;\kk)) 
\Rightarrow HH^{\bullet} (CF^*(L, L;\kk)) \]
In the case, $L =  (T_{Cl}, 2\zeta)$, this spectral sequence is necessarily trivial for degree reasons. Therefore, the quasi-equivalence (\ref{qiso}) gives us (see \cite[Equation 1.20]{flux}) that:
\[ HH^{\bullet} (CF^*( (L_{\Delta}, \zeta), (L_{\Delta}, \zeta);\kk)) = HH^{0} (CF^*( (L_{\Delta}, \zeta), (L_{\Delta}, \zeta);\kk)) = \kk \]

In view of this, the theory of deformations of $A_\infty$ algebras gives that the $A_\infty$ algebra on the cochain complex $CF^*( (L_{\Delta}, \zeta), (L_{\Delta}, \zeta);\kk)$ is formal (the obstruction classes \cite[Section 3]{seidelquartic} for trivializing the higher products vanish for grading reasons).

We conclude this discussion by deducing that the ring $HF^*( (L_{\Delta}, \zeta), (L_{\Delta}, \zeta;\kk)$ is semisimple:
\begin{thm} The $A_\infty$ algebra $CF^*( (L_{\Delta}, \zeta),
	(L_{\Delta}, \zeta);\kk)$ is quasi-isomorphic to the semisimple
	Clifford algebra $HF^*( (L_{\Delta}, \zeta), (L_{\Delta},
	\zeta);\kk) = \kk [x]/ (x^2 + c(\zeta) )$ where $c(\zeta) \neq 0$ and $x$ has degree 1.
\end{thm} 
\begin{proof} From the additive calculation of Floer cohomology
	\[ HF^*( (L_{\Delta}, \zeta), (L_{\Delta}, \zeta);\kk) = \kk
		\oplus \kk [1] \] we know that as a ring we have: \[ HF^*( (L_{\Delta},
		\zeta), (L_{\Delta}, \zeta);\kk) = \kk[x]/ (x^2+
	c(\zeta)) \] for some $c(\zeta) \in \kk$. The claimed result is to
	prove that $c(\zeta) \neq 0$. Suppose that $c(\zeta)=0$, then the Floer
	cohomology would be isomorphic to an exterior algebra $\kk[x]/(x^2)$.
	The $A_\infty$ algebra $CF^*( (L_{\Delta}, \zeta),
	(L_{\Delta}, \zeta);\kk)$ would then be equivalent (by homological
	perturbation \cite{kadeishvili}) to an $A_\infty$ structure on
	$\kk[x]/(x^2)$. The classification of such $A_\infty$ structures
	follows easily from deformation theory. It is explained in Example 3.20
	\cite{Smith} that they are given by a formal function \[W_k(x)= x^k +
	O(x^{k+1}) \text{ for } k \geq 3 \]and for the $A_\infty$ algebra
	$\mathcal{A}_k = (\kk[x]/(x^2), W_k)$, one has that
	$HH^*(\mathcal{A}_k, \mathcal{A}_k)$ has rank $k-1$, which is strictly
	greater than 1. On the other hand, we have seen above that
	$HH^{\bullet} (CF^*( (L_{\Delta}, \zeta), (L_{\Delta},
	\zeta);\kk))$ has rank 1. It follows then that $c(\zeta) \neq 0$ as
	required.  \end{proof}

It turns out that  $c(\zeta) = -\zeta^3$. This will be proved in the next section.

\section{Clifford module structure}\label{sct:cliff-mod}

In this section we will compute the Lagrangian intersection Floer cohomology of the Clifford torus $T_{Cl}$ with the Chiang Lagrangian $L_{\Delta}$. Note that for Floer cohomology to be defined the two Lagrangians must be equipped with local systems to give them the same $\mathfrak{m}_0$-value. In our earlier notation, we will fix a unit $\zeta\in(\ZZ/5)^\times$ and compute
\[HF^*((T_{Cl},2\zeta),(L_{\Delta},\zeta);\kk)\]
This is a $\ZZ/2$-graded module over the Clifford algebra $$HF^*((T_{Cl}, 2\zeta),(T_{Cl} , 2\zeta);\kk)=Cl^*(V,q_{2\zeta})$$ We will begin by recalling some basic facts on representations of Clifford algebras and we will finally deduce that the above module associated with $(L_{\Delta},\zeta)$ is quasi-isomorphic to the spin representation of the Clifford algebra $Cl^*(V,q_{2\zeta})$. 

\subsection{Preliminaries on representations of Clifford algebras}

Consider the Clifford algebra $Cl^*(V,q_{2\zeta})$ as a $\ZZ/2$-graded algebra. By \cite[Proposition 5.1]{ABS}, irreducible $\ZZ/2$-graded $Cl^*(V,q_{2\zeta})$-modules are in one-to-one correspondence with irreducible ungraded $Cl^0(V,q_{2\zeta})$-modules. This correspondence sends a graded module $M^*=M^0\oplus M^1$ to its even part $M^0$; the graded module is recovered as $Cl^*(V,q_{2\zeta})\otimes_{Cl^0}M^0$.

We can identify $Cl^0(V,q_{2\zeta})$ with a Clifford algebra on a
two-dimensional vector space by \cite[II.2.6]{Chev} and deduce by
\cite[II.2.1]{Chev} that $Cl^0(V,q_{2\zeta})$ is a central simple algebra. By the Artin-Wedderburn theorem this algebra is isomorphic to a two-by-two matrix algebra $M_2(\kk)$. In
particular, any module splits as a direct sum of simple modules and there is a
unique simple module, of rank two, which we call the spin representation $S$.
We write $S_{2\zeta}^*=Cl^*(V,q_{2\zeta})\otimes_{Cl^0}S$ for the unique simple
$\ZZ/2$-graded $Cl^*(V,q_{2\zeta})$-module, which has rank 4.

By \cite[II.2.6]{Chev} the centre of $Cl^*(V,q_{2\zeta})$ is two-dimensional and contains an odd element $z$ whose square is $-2D_{2\zeta}$ where $D_{2\zeta}=2\zeta^3$ is the discriminant of the quadratic form $q_{2\zeta}$. This central element spans the degree one module homomorphisms $S^*_{2\zeta}\to S_{2\zeta}^*[1]$. Indeed, we have:
$$ \OP{Hom}_{mod\text{-}Cl^*(V,q_{2\zeta})}(S_{2\zeta}^*,S_{2\zeta}^*)=\kk\oplus\kk z$$

In this ring, one has $z^2= \zeta^3$. We will momentarily show that $S_{2\zeta}^*$ and $(L_{\Delta},\zeta)$ are quasi-isomorphic in the Fukaya category, from which we will be able to deduce that $c({\zeta})= -\zeta^3 $.

\subsection{Computing the Floer cohomology}

\begin{lma} \label{circles} The Lagrangians $T_{Cl}$ and $L_{\Delta}$ intersect along a pair of
	circles.  \end{lma} \begin{proof} We work in coordinates
	$[u_0:\cdots:u_3]$ on $\PP(\OP{Sym}^3V)$ where $V$ is the standard
	representation of $SU(2)$. Recall from Section \ref{moment} that $L_\Delta$
	is defined by the following equations: 
\[ 3|u_0|^2+|u_1|^2-|u_2|^2-3|u_3|^2 = 0 \]
\[ \sqrt{3}u_0\bar{u}_1+2u_1\bar{u}_2+\sqrt{3}u_2\bar{u}_3 = 0 \]
Recall that the Clifford torus is given by
\[ |u_0| = |u_1| = |u_2| = |u_3| .\]
In the chart $u_0=1$, the Clifford torus consists of points
\[(u_0,\ldots,u_3)=(1,e^{-i\theta_1},e^{-i(\theta_1+\theta_2)},e^{-i(\theta_2+\theta_3)})\]
The intersection with $L_{\Delta}$ is the set of points for which
\[\sqrt{3}e^{i\theta_1}+2e^{i\theta_2}+\sqrt{3}e^{i\theta_3}=0.\]
We can rotate so that $\theta_2=0$; then $\theta_1=-\theta_3=\pm\cos^{-1}(1/\sqrt{3})$. Therefore the intersection consists of the two circles \[ (\theta_1,\theta_2,\theta_3)=(\pm\cos^{-1}(1/\sqrt{3})+\phi,\phi,\mp\cos^{-1}(1/\sqrt{3})+\phi)). \]
\end{proof}

\begin{cor} \label{pair}
We have
\[HF^*((T_{Cl}, 2\zeta ),(L_\Delta, \zeta);\kk)\cong S_{2\zeta}^*\]
as $Cl(V,q_{2\zeta})$-modules.
\end{cor}
\begin{proof}
Since $(T_{Cl},2\zeta)$ generates the summand of the Fukaya category containing
$(L_\Delta,\zeta)$ over the field $\kk$, this Floer cohomology group must be
non-zero. The corollary will follow from the classification of $\ZZ/2$-graded
$Cl^*(V,q_{2\zeta})$-modules if we can show that the rank of the Floer
cohomology is at most four-dimensional.

By Lemma \ref{circles}, the Clifford torus and the Chiang Lagrangian intersect along a pair of circles. After a small perturbation, using a perfect Morse function on each circle, they can be made to intersect at four points. This implies that the Floer cohomology is at most four-dimensional.
\end{proof}

\section{Generating the Fukaya category}\label{sct:gen-fuk}

We have seen above that $(T_{Cl},2\zeta)$ is a split-generator for the summand $\mathcal{F}(\cp{3}, 3 \zeta)$ of the Fukaya category of $\cp{3}$ and the $A_\infty$-structure on $HF^*((T_{Cl},2\zeta), (T_{Cl},2\zeta))$ is formal. This means that there is a quasi-equivalence between the derived categories: 
\begin{equation}\label{equivalence} D^b(mod\text{-}Cl^*(V, q_{2\zeta})) \simeq D^\pi (\mathcal{F}(\cp{3},3 \zeta)) \end{equation}

where the left hand side denotes bounded derived category of finitely generated modules over $Cl^*(V,q_{2\zeta})$ and the right hand side denotes the split-closure of a triangulated envelope of the summand of the Fukaya category $\mathcal{F}(\cp{3},3\zeta)$. 
This quasi-equivalence is a consequence of \cite[Corollary 4.9]{seidelbook} and the fact that the triangulated category $D^b(mod\text{-}Cl^*(V,q_{2\zeta}))$ is split-closed (\cite[Corollary 2.10]{balmerSchlichting}).

On the other hand, as we have seen in the previous section $Cl^*(V,q_{2\zeta})$
is a semisimple ring. In fact, $Cl^*(V, q_{2\zeta}) = S^*_{2\zeta} \oplus
S^*_{2\zeta}$, where $S^*_{2\zeta}$ is the unique simple module and any other
finitely generated module is isomorphic to a direct sum of finitely many
copies of $S^*_{2\zeta}$. In particular, $S^*_{2\zeta}$ is a generator of the
triangulated category $D^b(mod\text{-}Cl^*(V,q_{2\zeta}))$.

Now, by definition, there is a cohomologically full and faithful embedding of $\mathcal{F}(\cp{3}, 3\zeta))$ to $D^\pi(\mathcal{F}(\cp{3},3\zeta))$. Therefore, 
$(L_\Delta, \zeta)$ can be seen as an object of
$D^\pi(\mathcal{F}(\cp{3},3\zeta))$. On the other hand, we have seen in
Corollary \ref{pair} that the Floer cohomology $HF^*((T_{Cl},2\zeta ), (L_\Delta, \zeta))$ has rank
$4$. Therefore, under the above equivalence, $(L_\Delta,\zeta)$ should go to an
object of $D^b(mod\text{-}Cl^*(V,q_{2\zeta}))$ which has rank 4 as a
$Cl^*(V_,q_{2\zeta})$-module. There is a unique such module, namely
$S^*_{2\zeta}$. Therefore, we have obtained:

\begin{cor}\label{cor:summand} Under the quasi-equivalence \eqref{equivalence}, $(L_{\Delta},\zeta)$ is sent to an object quasi-isomorphic to $S^*_{2\zeta}$. In particular, \[ HF^*((L_{\Delta},\zeta), (L_{\Delta},\zeta)) \simeq Ext^*(S^*_{2\zeta}, S^*_{2\zeta}) \simeq \kk[x]/ (x^2- \zeta^3) \] 
	and $(L_{\Delta},\zeta)$ generates $D^\pi(\mathcal{F}(\cp{3},3\zeta))$ as a triangulated category.
\end{cor}

\bibliography{chiang}{}
\bibliographystyle{plain}

\end{document}